[1994/12/01]

\documentclass[draft]{amsart}

\usepackage{amsmath, amsfonts, amssymb}
\usepackage{amsthm,amscd,latexsym,mathrsfs,hyperref}
\usepackage{color}

\theoremstyle{plain} 
\newtheorem{theorem}{Theorem}[section]

\newtheorem{lemma}[theorem]{Lemma}
\newtheorem{proposition}[theorem]{Proposition}

\theoremstyle{definition}
\newtheorem{definition}{Definition}[section]
\newtheorem{remark}{Remark}[section]

\newtheorem{problem}{Problem}[section]

\DeclareMathOperator{\tr}{\mathrm{tr}}
\DeclareMathOperator{\spr}{\mathrm{spr}}
\DeclareMathOperator{\supp}{\mathrm{supp}}
\DeclareMathOperator{\argmin}{\mathrm{arg\, min}}

\begin{document}

\title[Existence and uniqueness of the $L^1$-Karcher mean]{Existence and uniqueness of the $L^1$-Karcher mean}
\author[Yongdo Lim and Mikl\'os P\'alfia]{Yongdo Lim and Mikl\'os P\'alfia}
\address{Department of Mathematics, Sungkyunkwan University, Suwon 440-746, Korea.}
\email{ylim@skku.edu}
\address{Department of Mathematics, Sungkyunkwan University, Suwon 440-746, Korea and Functional Analysis Research Group, Institute of Mathematics, University of Szeged, H-6720 Szeged, Hungary.}
\email{palfia.miklos@aut.bme.hu}


\date{\today}

\begin{abstract}
We extend the domain of the Karcher mean $\Lambda$ of positive operators on a Hilbert space to $L^1$-Borel probability measures on the cone of positive operators equipped with the Thompson part metric. We establish existence and uniqueness of $\Lambda$ as the solution of the Karcher equation and develop a nonlinear ODE theory for the relative operator entropy in the spirit of Crandall-Liggett, such that the solutions of the Karcher equation are stationary solutions of the ODE, and all generated solution curves enjoy the exponential contraction estimate. This is possible despite the facts that the Thompson metric is non-Euclidean, non-differentiable, non-commutative as a metric space as well as non-separable, and the positive cone is non-locally compact as a manifold. As further applications of the ODE approach, we prove the norm convergence conjecture of the power means of positive operators to the Karcher mean, and a Trotter-Kato product formula for the nonlinear semigroups explicitly expressed by compositions of two-variable geometric means. This can be regarded as a nonlinear continuous-time version of the law of large numbers.
\end{abstract}

\maketitle
\section{Introduction}
Let $\mathbb{S}$ denote the real vector space of bounded linear self-adjoint operators, $\mathbb{P}\subset\mathbb{S}$ denote the cone of positive definite operators on a Hilbert space $\mathcal{H}$ equipped with the operator norm $\|\cdot\|$. Let
\begin{equation*}
d_\infty(A,B):=\|\log(A^{-1/2}BA^{-1/2})\|=\spr(\log(A^{-1}B))
\end{equation*}
denote the Thompson metric, which turns $(\mathbb{P},d_\infty)$ into a complete metric space such
that the topology generated by $d_\infty$ agrees with the relative operator norm topology \cite{thompson}.
The Karcher mean \cite{karcher,sturm}, originally defined on $\mathbb{P}$ for finite dimensional $\mathcal{H}$ \cite{bhatiaholbrook,moakher} as a non-commutative generalization of the
geometric mean \cite{kubo}, has been intensively investigated in the last decade \cite{bhatiakarandikar,limpalfia,lawsonlim1}. For a $k$-tuple of operators $\mathbb{A}:=(A_1,\ldots, A_k)$ with corresponding weight $\omega:=(w_1,\ldots,w_k)$ where $A_i\in\mathbb{P}$, $w_i>0$ and $\sum_{i=1}^kw_i=1$ the Karcher mean $\Lambda(\omega,\mathbb{A})$ is defined as the unique solution of the Karcher equation
\begin{equation}\label{eq:intro1}
\sum_{i=1}^kw_i\log(X^{-1/2}A_iX^{-1/2})=0
\end{equation}
for $X\in\mathbb{P}$. The existence and uniqueness in the infinite dimensional case was proved by Lawson-Lim \cite{lawsonlim1}, generalizing the approximation technique of power means given in the finite dimensional case by Lim-P\'alfia \cite{limpalfia}. The power mean $P_t(\omega,\mathbb{A})$ for $t\in(0,1]$ is defined as the unique solution of the operator equation
\begin{equation}\label{eq:intro2}
\sum_{i=1}^kw_iX\#_tA_i=X
\end{equation}
where $A\#_tB=A^{1/2}\left(A^{-1/2}BA^{-1/2}\right)^{t}A^{1/2}$ is the geometric mean of $A,B\in\mathbb{P}$. It is proved in \cite{lawsonlim1} that $t\mapsto P_t(\omega,\mathbb{A})$ is a decreasing sequence in the strong operator topology, strong operator converging to $\Lambda(\omega,\mathbb{A})$, extending the result of \cite{limpalfia}. A further generalization of $\Lambda$ to Borel probability measures with bounded support has been done in \cite{kimlee,palfia2} by integrating with respect to a Borel probability measure $\mu$ in \eqref{eq:intro1} and \eqref{eq:intro2} instead of taking sums.

Let $\mathcal{P}^1(\mathbb{P})$ denote the convex set of $\tau$-additive Borel probability measures
  $\mu$ on $(\mathbb{P},\mathcal{B}(\mathbb{P}))$ such that $\int_{\mathbb{P}}d_\infty(X,A)d\mu(A)<+\infty$
   for all $X\in\mathbb{P}$. Recall that a Borel measure $\mu$ is $\tau$-\emph{additive}
 if $\mu(\bigcup_\alpha U_\alpha)=\sup_\alpha \mu(U_\alpha)$ for all directed families $\{U_\alpha: \alpha\in D\}$
 of open sets. In this paper for $\mu\in\mathcal{P}^1(\mathbb{P})$ we consider the operator equation
\begin{equation}\label{eq:intro3}
\int_{\mathbb{P}}\log(X^{-1/2}A_iX^{-1/2})d\mu(A)=0,
\end{equation}
and establish the existence and uniqueness of the solution for $X\in\mathbb{P}$ which provides the extension of the map $\Lambda(\cdot)$ to the case of $L^1$-probability measures over the infinite dimensional cone $(\mathbb{P},d_\infty)$. In particular existence is established by approximation with finitely supported measures in the $L^1$-Wasserstein distance, by extending the fundamental $L^1$-Wasserstein contraction property
\begin{equation*}
d_\infty(\Lambda(\mu),\Lambda(\nu))\leq W_1(\mu,\nu)
\end{equation*}
for any $\mu,\nu\in\mathcal{P}^1(\mathbb{P})$ originally introduced by Sturm on $\mathrm{CAT}(0)$-spaces \cite{sturm}. We also prove the uniqueness of solution of \eqref{eq:intro3}, and develop a nonlinear ODE theory for $\Lambda$ by considering the Cauchy problem
\begin{equation}\label{eq:introCP}
\begin{split}
X(0)&:=X\in\mathbb{P},\\
\dot{X}(t)&=\int_{\mathbb{P}}\log_{X(t)}Ad\mu(A),
\end{split}
\end{equation}
for $t\in[0,\infty)$, where $\log_XA=X^{1/2}\log\left(X^{-1/2}AX^{-1/2}\right)X^{1/2}$ is the relative operator entropy \cite{fujii,fujiiseo}. We prove that the solutions of \eqref{eq:introCP} can be constructed by a discrete backward Euler-scheme converging in the $d_\infty$ distance, generalizing the classical Crandall-Liggett techniques developed in Banach spaces \cite{crandall}. In order to obtain the discretizations, we introduce the nonlinear resolvent
\begin{equation}\label{eq:introResolvent}
J_{\lambda}^{\mu}(X):=\Lambda\left(\frac{\lambda}{\lambda+1}\mu+\frac{1}{\lambda+1}\delta_X\right)
\end{equation}
for $\lambda>0$. The advantage is that $J_{\lambda}^{\mu}$ is a strict contraction with respect to $d_\infty$, and satisfies the \emph{resolvent identity} necessary to obtain the $O(\sqrt{\lambda})$ convergence rate estimate of the exponential formula to the solution of \eqref{eq:introCP} along with the semigroup property in \cite{crandall}. We further obtain the exponential contraction rate estimate
\begin{equation*}
d_\infty(\gamma(t),\eta(t))\leq e^{-t}d_\infty(\gamma(0),\eta(0))
\end{equation*}
valid for two solution curves of \eqref{eq:introCP} with varying initial points. This large time behavior also ensures the uniqueness of stationary points, and the uniqueness of the solution of \eqref{eq:intro3}.

Furthermore, using the same exponential contraction estimate, we perform additional analysis of the Frech\'et-differential of the left hand side of \eqref{eq:intro3} at the unique solution $\Lambda(\mu)$, eventually proving the norm convergence conjecture of power means $P_t$ to $\Lambda$ as $t\to 0$, a problem first mentioned in \cite{lawsonlim1} as a possible strengthening of the strong operator convergence. From the analysis of the Frech\'et-differential, a resolvent convergence also follows which leads to a continuous-time Trotter-Kato-type product formula that is closely related to the law of large numbers of Sturm \cite{sturm} and its deterministic "nodice" counterparts proved in \cite{Hol,limpalfia2} valid in $\mathrm{CAT}(0)$-spaces. In particular we prove that for a sequence $\mu_n$ of finitely supported probability measures $W_1$-convergent to $\mu$, we have
\begin{equation*}
\lim_{n\to\infty}S^{\mu_n}(t)=S^{\mu}(t)
\end{equation*}
in $d_\infty$, where $S^{\mu}(t)$ denotes the solution of the Cauchy problem \eqref{eq:introCP} corresponding to $\mu$. Under the assumption $\mu_n=\sum_{i=1}^n\frac{1}{n}\delta_{Y_i}$, we also prove the explicit product formula
\begin{equation*}
\lim_{m\to\infty}\left(F^{\mu_n}_{t/m}\right)^m=S^{\mu_n}(t)
\end{equation*}
in $d_\infty$, where $F^{\mu_n}_{\rho}:=J_{\rho/n}^{\delta_{Y_n}}\circ\cdots\circ J_{\rho/n}^{\delta_{Y_1}}$ with $J_{\rho}^{\delta_{A}}(X):=X\#_{\frac{\rho}{\rho+1}}A$ in the spirit of \eqref{eq:introResolvent}. The above formula is advantageous, since it only contains iterated geometric means of only two operators, hence explicitly calculable.

It must be noted that although similar results are available for the Karcher mean in the $\mathrm{CAT}(0)$ \cite{bacak,stojkovic,sturm} or even $\mathrm{CAT}(k)$ \cite{ohta0,ohta} setting, all these techniques break down in the infinite dimensional case $(\mathbb{P},d_\infty)$ due to the nonexistence of a convex potential function so that the solutions of \eqref{eq:introCP} is a gradient flow, the non-differentiability of the squared distance function $d_\infty^2(A,\cdot)$ and the "non-commutativity" in the sense of \cite{ohta} of the operator norm $\|\cdot\|$ appearing in the formula for $d_\infty$. Even proving the resolvent convergence necessary for a Trotter-Kato product formula is non-trivial in the infinite dimensional case due to the lack of local compactness of the manifold $\mathbb{P}$.

The paper is organized as follows. In section 2 we gather all the necessary information available for the Karcher mean of finitely supported measures in relation with the distance $d_\infty$ and the $L^1$-Wasserstein distance $W_1$, and in section 3 we extend the domain of $\Lambda$ to $\mathcal{P}^1(\mathbb{P})$ by $W_1$-continuity and in section 4 we prove its uniqueness as a solution of \eqref{eq:intro3}. In section 5 we develop the ODE theory corresponding to \eqref{eq:introCP} by generalizing the argumentation in \cite{crandall}. In section 6 we develop the theory of approximation semigroups for the ODE flow \eqref{eq:introCP}. In section 7 we establish the $d_\infty$ convergence of the approximating resolvents necessary for establishing the Trotter-Kato formula of section 6 by combining the large time behavior of the solutions of \eqref{eq:introCP} with operator theoretic techniques. As a byproduct we prove the norm convergence conjecture of $P_t$ to $\Lambda$ as $t\to 0$. The last section gathers the consequences of the earlier sections, establishing a continuous-time result corresponding to the law of large numbers.

\section{Preliminaries}
Let $\mathcal{P}^1(\mathbb{P})$ denote the convex set of $\tau$-additive Borel probability measures
  $\mu$ on $(\mathbb{P},\mathcal{B}(\mathbb{P}))$ such that $\int_{\mathbb{P}}d_\infty(X,A)d\mu(A)<+\infty$
   for all $X\in\mathbb{P}$. Notice that $\int_{\mathbb{P}}d_\infty(X,A)d\mu(A)<+\infty$ implies by Proposition 23 of Chapter 4 in \cite{royden}, the \emph{uniform intergability} of $d_\infty$, that is
\begin{equation}\label{eq:UnifromIntegrable}
\lim_{R\to\infty}\int_{d_\infty(X,A)\geq R}d_\infty(X,A)d\mu(A)=0.
\end{equation}
More generally we say that a sequence $\mu_n\in\mathcal{P}^1(\mathbb{P})$ is uniformly integrable if 
\begin{equation*}
\lim_{R\to \infty}\limsup_{n\to\infty}\int_{d_\infty(X,A)\geq R}d_\infty(X,A)d\mu_n(A)=0
\end{equation*}
for a (thus all) $X\in\mathbb{P}$.
   
   For $\tau$-additive measures one may realize the complement of the support as a directed union of open sets of measure $0$
 hence the complement has measure $0$, and the support has measure $1$. For the separability of the support of $\sigma$-additive measures over metric spaces, see \cite{lawson}.

\begin{proposition}\label{P:separableSupp}
Let $\mu\in\mathcal{P}^1(\mathbb{P})$. Then the support $\supp(\mu)$ is separable.
\end{proposition}

The $L^1$-Wasserstein distance between $\mu,\nu\in\mathcal{P}^1(\mathbb{P})$ is defined as
\begin{equation*}
W_1(\mu,\nu)=\inf_{\gamma\in\Pi(\mu,\nu)}\int_{\mathbb{P}\times\mathbb{P}}d_\infty(A,B)d\gamma(A,B)
\end{equation*}
where $\Pi(\mu,\nu)$ denotes the set of all $\tau$-additive Borel probability measures on the product space $\mathbb{P}\times\mathbb{P}$ with marginals $\mu$ and $\nu$. We consider $\tau$-additive measures, since the following is not true in general for
   $\sigma$-additive Borel probability measures, however it holds for $\tau$-additive ones:

\begin{proposition}[Theorem 8.3.2. \& Example 8.1.6. \cite{bogachev}]\label{P:weakW1agree}
The topology generated by the Wasserstein metric $W_1(\cdot,\cdot)$ on $\mathcal{P}^1(\mathbb{P})$ agrees with the weak-$*$ (also called weak) topology of $\mathcal{P}^1(\mathbb{P})$ on uniformly integrable sequences of probability measures. Moreover finitely supported probability measures are $W_1$-dense in $\mathcal{P}^1(\mathbb{P})$.
\end{proposition}
\begin{proof}
Since the support of any member of $\mathcal{P}^1(\mathbb{P})$ is separable and for a $\tau$-additive probability measure its support has measure $1$, the proofs of Kantorovich duality go through when restricted to the supports of $\mu,\nu\in\mathcal{P}^1(\mathbb{P})$ thus basically arriving at the Polish metric space case, see for example Theorem 6.9 in \cite{villani} and Theorem 8.10.45 in \cite{bogachev}. In particular Theorem 6.9 in \cite{villani} proves the equivalence between the two topologies.

Then by Varadarajan's theorem which can be found as Theorem 11.4.1. in \cite{dudley} we have that for any $\mu\in\mathcal{P}^1(\mathbb{P})$ the empirical probability measures $\mu_n:=\sum_{i=1}^n\frac{1}{n}\delta_{Y_i}$ converge weakly to $\mu$ almost surely on the Polish metric space $(\supp(\mu),d_\infty)$, where $Y_i$ is a sequence of i.i.d. random variables on the Polish metric space $(\supp(\mu),d_\infty)$ with law $\mu$. 
So for each bounded continuous function $f$ on $(\supp(\mu),d_\infty)$ we have $\int_{\supp(\mu)}fd\mu_n\to\int_{\supp(\mu)}fd\mu$ which happens outside of a set of measure $0$. 
So on the complement we have weak convergence of $\mu_n$ to $\mu$. Now, one is left with checking that $\mu_n$ is a uniformly integrable sequence which follows from the uniform integrability of $\mu$ itself.
\end{proof}


\begin{definition}[strong measurability, Bochner integral]\label{D:BochnerIntegrable}
Let $(\Omega,\Sigma,\mu)$ be finite measure space and let
$f:\Omega\mapsto\mathbb{P}$. Then $f$ is \emph{strongly measurable}
if there exists a sequence of simple functions $f_n$, such that
$\lim_{n\to\infty}f_n(\omega)=f(\omega)$ almost surely.

The function $f:\Omega\mapsto\mathbb{P}$ is \emph{Bochner integrable} if the following are satisfied:
\begin{itemize}
\item[(1)] $f$ is strongly measurable;
\item[(2)] there exists a sequence of simple functions $f_n$, such that $\lim_{n\to\infty}\int_{\Omega}\|f(\omega)-f_n(\omega)\|d\mu(\omega)=0$
\end{itemize}
In this case we define the \emph{Bochner integral} of $f$ by
$$\int_{\Omega}f(\omega)d\mu(\omega):=\lim_{n\to\infty}\int_{\Omega}f_n(\omega)d\mu(\omega).$$
\end{definition}
It is well known that a strongly measurable function $f$ on a finite measure space $(\Omega,\Sigma,\mu)$
is Bochner integrable if and only if $\int_{\Omega}\|f(\omega)\|d\mu(\omega)<\infty$.

The logarithm map $\log:{\Bbb P}\to \mathbb{S}$ is
differentiable and  is contractive from the exponential metric
increasing (EMI) property (\cite{LL07})
\begin{eqnarray}\label{E:EMI}||\log X-\log Y||\leq d_{\infty}(X,Y), \ \ \
X,Y\in {\Bbb P}. \end{eqnarray} This property reflects the
seminegative curvature of the Thompson metric, which can be realized
as a Banach-Finsler metric arising from the Banach space norm on
$\mathbb{S}$: for $A\in {\Bbb P},$ the Finsler norm of $X\in T_A
{\Bbb P}=\mathbb{S}$ is given by $\Vert X\Vert_{A}= \Vert
A^{-1/2}XA^{-1/2}\Vert$ and the exponential and logarithm maps are
\begin{eqnarray}\label{E:BF}
\exp_A(X)&=&A^{1/2}\exp(A^{-1/2}XA^{-1/2})A^{1/2},\\
\log_A(X)&=&A^{1/2}\log(A^{-1/2}XA^{-1/2})A^{1/2}.
\end{eqnarray}

Notice that also $\log_AX=A\log(A^{-1}X)$.

\begin{lemma}\label{L:LogIntegrable}
For all $\mu\in\mathcal{P}^1(\mathbb{P})$ and $X\in\mathbb{P}$, the Bochner integral $\int_{\mathbb{P}}\log_{X}Ad\mu(A)$ exists.
\end{lemma}
\begin{proof}
First of all, notice that $A\mapsto X\log(X^{-1}A)$ is strongly measurable. Indeed since $A\mapsto X\log(X^{-1}A)$ is norm continuous, hence $d_\infty$ to norm continuous and it is almost separably valued, thus by the Pettis measurability theorem it is strongly measurable. Then
\begin{equation*}
\begin{split}
\int_{\mathbb{P}}\|X\log(X^{-1}A)\|d\mu(A)&\leq \int_{\mathbb{P}}\|X^{1/2}\|\|\log(X^{-1/2}AX^{-1/2})\|\|X^{1/2}\|d\mu(A)\\
&=\|X\|\int_{\mathbb{P}}\|\log(X^{-1/2}AX^{-1/2})\|d\mu(A)\\
&=\|X\|\int_{\mathbb{P}}d_\infty(X,A)d\mu(A)<\infty
\end{split}
\end{equation*}
which shows Bochner integrability.
\end{proof}

\begin{definition}[Karcher equation/mean]\label{D:Karcher}
For a $\mu\in\mathcal{P}^1(\mathbb{P})$ the \emph{Karcher equation} is defined as
\begin{equation}\label{eq:D:Karcher}
\int_{\mathbb{P}}\log_XAd\mu(A)=0,
\end{equation}
where $X\in\mathbb{P}$. If \eqref{eq:D:Karcher} has a unique solution in $X\in\mathbb{P}$,
then it is called the Karcher mean and is denoted by $\Lambda(\mu)$.
\end{definition}

\begin{definition}[Weighted geometric mean]\label{D:GeometricMean}
Let $A,B\in\mathbb{P}$ and $t\in[0,1]$. Then for $(1-t)\delta_A+t\delta_B=:\mu\in\mathcal{P}^1(\mathbb{P})$ the Karcher equation
\begin{equation*}
\int_{\mathbb{P}}\log_XAd\mu(A)=(1-t)\log_XA+t\log_XB=0
\end{equation*}
has a unique solution $A\#_tB=\Lambda(\mu)$ called the \emph{weighted geometric mean} and
\begin{equation*}
A\#_tB=A^{1/2}\left(A^{-1/2}BA^{-1/2}\right)^tA^{1/2}=A\left(A^{-1}B\right)^t.
\end{equation*}
\end{definition}

By the dominated convergence theorem and Lemma~\ref{L:LogIntegrable} we have the following:
\begin{lemma}\label{L:gradContinuous}
For each $X\in\mathbb{P}$ and $\mu\in\mathcal{P}^1(\mathbb{P})$ the function $X\mapsto \int_{\mathbb{P}}\log_XAd\mu(A)$ is $d_\infty$ to norm continuous.
\end{lemma}
\begin{proof}
Pick a sequence $X_n\to X$ in the $d_\infty$ topology in $\mathbb{P}$. Then
\begin{equation}\label{eq1:L:gradContinuous}
\begin{split}
&\left\|\int_{\mathbb{P}}\log_{X_n}Ad\mu(A)-\int_{\mathbb{P}}\log_{X}Ad\mu(A)\right\|\\
&\leq\int_{\mathbb{P}}\left\|\log_{X_n}A-\log_{X}A\right\|d\mu(A)\\
&\leq\int_{\mathbb{P}}\left\|\log_{X_n}A\right\|+\left\|\log_{X}A\right\|d\mu(A)\\
&\leq\|X_n\|\int_{\mathbb{P}}d_\infty(X_n,A)d\mu(A)+\|X\|\int_{\mathbb{P}}d_\infty(X,A)d\mu(A)<\infty,
\end{split}
\end{equation}
thus $\left\|\log_{X_n}A-\log_{X}A\right\|$ is integrable.
Since $d_\infty$ agrees with the relative norm
topology, we have that
\begin{equation*}
F_n(A):=\left\|\log_{X_n}A-\log_{X}A\right\|\to 0
\end{equation*}
point-wisely for every $A\in\mathbb{P}$ as $n\to\infty$. Then by the dominated convergence theorem we obtain
\begin{equation*}
\begin{split}
\lim_{n\to\infty}\int_{\mathbb{P}}\left\|\log_{X_n}A-\log_{X}A\right\|d\mu(A)&=\int_{\mathbb{P}}\lim_{n\to\infty}\left\|\log_{X_n}A-\log_{X}A\right\|d\mu(A)\\
&=0.
\end{split}
\end{equation*}
In view of \eqref{eq1:L:gradContinuous} this proves the assertion.
\end{proof}

For some further known facts below, see for example \cite{lawsonlim1}.

\begin{lemma}\label{L:geometricMeanContraction}
For a fixed $A\in\mathbb{P}$ and $t\in[0,1]$ the function $f(X):=A\#_tX$ is a contraction on $(\mathbb{P},d_\infty)$ with Lipschitz constant $(1-t)$.
\end{lemma}

\begin{proposition}[Power means, cf. \cite{lawsonlim1,limpalfia}]\label{P:PowerMeans}
Let $t\in(0,1]$, $A_i\in\mathbb{P}$ for $1\leq i\leq n$ and let
$\omega=(w_1,\ldots,w_n)$ be a probability vector so that
$\mu=\sum_{i=1}^nw_i\delta_{A_i}\in\mathcal{P}^1(\mathbb{P})$. Then
the function \begin{equation*} f(X):=\int_{\mathbb{P}}X\#_tAd\mu(A)
\end{equation*}
is a contraction on $(\mathbb{P},d_\infty)$ with Lipschitz constant
$(1-t)$, and thus the operator equation
\begin{equation}\label{eq:P:PowerMeans}
\int_{\mathbb{P}}X\#_tAd\mu(A)=X
\end{equation}
has a unique solution in $X\in\mathbb{P}$ which is denoted by $P_t(\mu)$.
\end{proposition}

\begin{theorem}[see \cite{lawsonlim1,limpalfia}]\label{T:PowerMeanLimit}
Let $t\in(0,1]$, $A_i\in\mathbb{P}$ for $1\leq i\leq n$ and let $\omega=(w_1,\ldots,w_n)$ be a probability vector so that $\mu=\sum_{i=1}^nw_i\delta_{A_i}\in\mathcal{P}^1(\mathbb{P})$. Then for $1\geq s\geq t>0$ we have $P_s(\mu)\geq P_t(\mu)$ and the strong operator limit
\begin{equation}\label{eq:T:PowerMeanLimit}
X_0:=\lim_{t\to 0+}P_t(\mu)
\end{equation}
exists and $X_0=\Lambda(\mu)$.
\end{theorem}

By the above we define $P_0(\mu):=\Lambda(\mu)$.

\begin{proposition}[\cite{lawsonlim1,limpalfia}]\label{P:PowerPoperties}
The function $P_t(\cdot)$ is operator monotone. That is, let
$t\in[0,1]$, $A_i\leq B_i\in\mathbb{P}$ for $1\leq i\leq n$ and let
$\omega=(w_1,\ldots,w_n)$ be a probability vector so that
$\mu=\sum_{i=1}^nw_i\delta_{A_i},\nu=\sum_{i=1}^nw_i\delta_{B_i}\in\mathcal{P}^1(\mathbb{P})$.
Then
\begin{equation}\label{eq:P:PowerPoperties}
P_t(\mu)\leq P_t(\nu).
\end{equation}
\end{proposition}

\begin{theorem}[see Theorem 6.4. \cite{lawsonlim1}]\label{T:KarcherExist}
Let $A_i\in\mathbb{P}$ for $1\leq i\leq n$ and let $\omega=(w_1,\ldots,w_n)$ be a probability vector. Then for $\mu=\sum_{i=1}^nw_i\delta_{A_i}$ the equation \eqref{eq:D:Karcher} has a unique positive definite solution $\Lambda(\mu)$.

In the special case $n=2$, we have
\begin{equation}\label{eq:T:KarcherExist}
\Lambda((1-t)\delta_{A}+t\delta_{B})=A\#_tB
\end{equation}
for any $t\in[0,1]$, $A,B\in\mathbb{P}$.
\end{theorem}

\begin{proposition}[see Proposition 2.5. \cite{lawsonlim1}]\label{P:KarcherW1contracts}
Let $A_i,B_i\in\mathbb{P}$ for $1\leq i\leq n$. Then $\Lambda$ for $\mu=\frac{1}{n}\sum_{i=1}^n\delta_{A_i}$ and $\nu=\frac{1}{n}\sum_{i=1}^n\delta_{B_i}$ satisfies
\begin{equation}\label{eq0:P:KarcherW1contracts}
d_{\infty}(\Lambda(\mu),\Lambda(\nu))\leq\sum_{i=1}^n\frac{1}{n}d_\infty(A_i,B_i),
\end{equation}
in particular by permutation invariance of $\Lambda$ in the variables $(A_1,\ldots,A_n)$ we have
\begin{equation}\label{eq:P:KarcherW1contracts}
d_{\infty}(\Lambda(\mu),\Lambda(\nu))\leq\min_{\sigma\in S_n}\sum_{i=1}^n\frac{1}{n}d_\infty(A_i,B_{\sigma(i)})=W_1(\mu,\nu).
\end{equation}
\end{proposition}

\section{Extension of $\Lambda$ by $W_1$-continuity}
We extend $\Lambda$ and its contraction properties by using continuity and contraction property of it with respect to $W_1$, along with the approximation properties of $\mathcal{P}^1(\mathbb{P})$ with respect to the metric $W_1$.

\begin{lemma}\label{L:converge}
Let $X,Y\in\mathbb{P}$ and
$\mu_n,\mu\in\mathcal{P}^1(\mathbb{P})$ and
$\supp(\mu_n),\supp(\mu)\subseteq Z\subset\mathbb{P}$ where $Z$ is
closed and separable. Assume also that $X\to Y$ in
$d_\infty$, $\mu_n\to \mu$ in $W_1$. Then
\begin{equation*}
\int_{\mathbb{P}}\log_{X}Ad\mu_n(A)\to \int_{\mathbb{P}}\log_{Y}Ad\mu(A)
\end{equation*}
in the weak Banach space topology.
\end{lemma}
\begin{proof}
Let $x,y\in\mathbb{P}$ and $\mu,\nu\in\mathcal{P}^1(\mathbb{P})$. For any real-valued norm continuous linear functional $l^*$ we have
\begin{equation}\label{eq1:L:converge}
\begin{split}
&\left|\left\langle \int_{\mathbb{P}}\log_{x}Ad\mu_n(A)-\int_{\mathbb{P}}\log_{y}Ad\mu(A),l^* \right\rangle\right| \\
&\leq \left|\left\langle \int_{\mathbb{P}}\log_{x}Ad\mu(A)-\int_{\mathbb{P}}\log_{y}Ad\mu(A),l^* \right\rangle\right|\\
&\quad +\left|\left\langle \int_{\mathbb{P}}\log_{x}Ad\mu_n(A)-\int_{\mathbb{P}}\log_{x}Ad\mu(A),l^* \right\rangle\right|\\
&\leq \left|\left\langle \int_{\mathbb{P}}\log_{x}Ad\mu(A),l^* \right\rangle-\left\langle\int_{\mathbb{P}}\log_{y}Ad\mu(A),l^* \right\rangle\right|\\
&\quad +\left|\int_{\mathbb{P}}\left\langle\log_{x}A,l^* \right\rangle d\mu_n(A)-\int_{\mathbb{P}}\left\langle\log_{x}A,l^* \right\rangle d\mu(A)\right|.
\end{split}
\end{equation}
If $x\to y$ in $d_\infty$, then the first term in the above converges to $0$ by Lemma~\ref{L:gradContinuous}. Now the complete metric space $(Z,d_\infty)$ is separable, so we can apply some well known theorems for the metric $W_1$ restricted for probability measures with support included in $Z$. In fact Proposition 7.1.5 in \cite{AGS} tells us that $\mu_n$ has uniformly integrable $1$-moments, i.e.
\begin{equation*}
\lim_{R\to \infty}\sup_{n}\int_{d_\infty(x,A)\geq R}d_\infty(x,A)d\mu_n(A)=0.
\end{equation*}
Now, for the second term in \eqref{eq1:L:converge} we have the estimates
\begin{equation*}
\begin{split}
\int_{\mathbb{P}}\left|\left\langle\log_{x}A,l^* \right\rangle\right| d\mu_n(A)&\leq \|l^*\|_*\|x\|\int_{\mathbb{P}}\|\log(x^{-1/2}Ax^{-1/2})\|d\mu_n(A)\\
&\leq \|l^*\|_*\|x\|\int_{\mathbb{P}}d_\infty(x,A)d\mu_n(A),
\end{split}
\end{equation*}
which means that $\int_{\mathbb{P}}\left|\left\langle\log_{x}A,l^* \right\rangle\right| d\mu_n(A)$ is uniformly integrable as well. In \cite{bogachev} Lemma 8.4.3. says that if $\xi_\alpha \to\xi$ in the weak-$*$ topology for Baire probability measures $\xi_\alpha,\xi$ on a topological space $X$, then for every real-valued continuous function $f$ on $X$ satisfying $\lim_{R\to \infty}\sup_{\alpha}\int_{|f|\geq R}|f|d\xi_\alpha=0,$ we have $\lim_{\alpha}\int_{X}fd\xi_\alpha=\int_{X}fd\xi$. Thus the second term of \eqref{eq1:L:converge} also converges to $0$.
\end{proof}

\begin{theorem}\label{T:LambdaExists}
For all $\mu\in\mathcal{P}^1(\mathbb{P})$ there exists a solution of \eqref{eq:D:Karcher} denoted by $\Lambda(\mu)$ (with slight abuse of notation), which satisfies
\begin{equation}\label{eq:T:LambdaExists}
d_\infty(\Lambda(\mu),\Lambda(\nu))\leq W_1(\mu,\nu)
\end{equation}
for all $\nu\in\mathcal{P}^1(\mathbb{P})$.
\end{theorem}
\begin{proof}
Let $\mu\in\mathcal{P}^1(\mathbb{P})$. Then by Proposition~\ref{P:weakW1agree} there exists a $W_1$-convergent sequence of finitely supported probability measures $\mu_n\in\mathcal{P}^1(\mathbb{P})$ such that $W_1(\mu,\mu_n)\to 0$. By Theorem~\ref{T:KarcherExist} $\Lambda(\mu_n)$ exists for any $n$ in the index set. We also have that $W_1(\mu_m,\mu_n)\to 0$ as $m,n\to\infty$ and by \eqref{eq:P:KarcherW1contracts} it follows that $d_\infty(\Lambda(\mu_m),\Lambda(\mu_n))\to 0$ as $m,n\to\infty$, i.e. $\Lambda(\mu_n)$ is a $d_\infty$ Cauchy sequence. Thus we define
\begin{equation}\label{Q}
\tilde{\Lambda}(\mu):=\lim_{n\to\infty}\Lambda(\mu_n).
\end{equation}
Since \eqref{eq:T:LambdaExists} holds by Proposition~\ref{P:KarcherW1contracts}
 for finitely supported probability measures, we extend \eqref{eq:T:LambdaExists}
  to the whole of $\mathcal{P}^1(\mathbb{P})$ by $W_1$-continuity, using the $W_1$-density of finitely supported probability measures in $\mathcal{P}^1(\mathbb{P})$.

Then by construction for all $n$ we have
\begin{equation*}
\int_{\mathbb{P}}\log_{\Lambda(\mu_n)}Ad\mu_n(A)=0,
\end{equation*}
thus by Lemma~\ref{L:converge} we have
\begin{equation*}
\int_{\mathbb{P}}\log_{\Lambda(\mu_n)}Ad\mu_n(A)\to \int_{\mathbb{P}}\log_{\tilde{\Lambda}(\mu)}Ad\mu(A)
\end{equation*}
weakly, that is
\begin{equation*}
\int_{\mathbb{P}}\log_{\tilde{\Lambda}(\mu)}Ad\mu(A)=0.
\end{equation*}
\end{proof}

\begin{definition}[Karcher mean]\label{D:KarcherMean}
Given a $\mu\in\mathcal{P}^1(\mathbb{P})$, we define $\Lambda(\mu)$ as the limit obtained in Theorem~\ref{T:LambdaExists}. Notice that the limit does not depend on the actual approximating sequence of measures due to \eqref{eq:T:LambdaExists}.
\end{definition}

\section{The uniqueness of $\Lambda$}

In this section we establish the uniqueness of the solution of \eqref{eq:D:Karcher}. We will need the following result that establishes this for probability measures with bounded support.

\begin{theorem}[Theorem 6.13. \& Example 6.1. in \cite{palfia2}]\label{T:KarcherExist2}
Let $\mu\in\mathcal{P}^1(\mathbb{P})$ such that $\supp(\mu)$ is bounded. Then the Karcher equation \eqref{eq:D:Karcher} has a unique positive definite solution $\Lambda(\mu)$.
\end{theorem}

The following result is well known for Wasserstein spaces over general metric spaces, we provide its proof for completeness.
\begin{proposition}\label{P:W_1convex}
The $W_1$ distance is convex, that is for $\mu_1,\mu_2,\nu_1,\nu_2\in\mathcal{P}^1(\mathbb{P})$ and $t\in[0,1]$ we have
\begin{equation}\label{eq:P:W_1convex}
W_1((1-t)\mu_1+t\mu_2,(1-t)\nu_1+t\nu_2)\leq (1-t)W_1(\mu_1,\nu_1)+tW_1(\mu_2,\nu_2).
\end{equation}
\end{proposition}
\begin{proof}
Let $\omega_1\in\Pi(\mu_1,\nu_1), \omega_2\in\Pi(\mu_2,\nu_2)$ where $\Pi(\mu,\nu)\subseteq\mathcal{P}(\mathbb{P}\times\mathbb{P})$ denote the set of all couplings of $\mu,\nu\in\mathcal{P}^1(\mathbb{P})$. Then $(1-t)\omega_1+t\omega_2\in\Pi((1-t)\mu_1+t\mu_2,(1-t)\nu_1+t\nu_2)$ and we have
\begin{equation*}
\begin{split}
W_1&((1-t)\mu_1+t\mu_2,(1-t)\nu_1+t\nu_2)\\
&=\inf_{\gamma\in\Pi((1-t)\mu_1+t\mu_2,(1-t)\nu_1+t\nu_2)}\int_{\mathbb{P}\times\mathbb{P}}d_\infty(A,B)d\gamma(A,B)\\
&\leq\int_{\mathbb{P}\times\mathbb{P}}d_\infty(A,B)d((1-t)\omega_1+t\omega_2)(A,B)\\
&=(1-t)\int_{\mathbb{P}\times\mathbb{P}}d_\infty(A,B)d\omega_1(A,B)+t\int_{\mathbb{P}\times\mathbb{P}}d_\infty(A,B)d\omega_2(A,B),
\end{split}
\end{equation*}
thus by taking infima in $\omega_1\in\Pi(\mu_1,\nu_1), \omega_2\in\Pi(\mu_2,\nu_2)$ \eqref{eq:P:W_1convex} follows.
\end{proof}

\begin{theorem}\label{T:L1KarcherUniqueness}
Let $\mu\in\mathcal{P}^1(\mathbb{P})$. Then the Karcher equation \eqref{eq:D:Karcher} has a unique solution in $\mathbb{P}$.
\end{theorem}
\begin{proof}
Let $X\in\mathbb{P}$ be a solution of \eqref{eq:D:Karcher}, i.e.
\begin{equation*}
\int_{\mathbb{P}}\log_{X}Ad\mu(A)=0.
\end{equation*}
Let $B(X,R):=\{Y\in\mathbb{P}:d_\infty(Y,X)<R\}$. Then since $\int_{\mathbb{P}}d_\infty(X,A)d\mu(A)<+\infty$ from Proposition 23 of Chapter 4 in \cite{royden} it follows that
\begin{equation}\label{eq0:T:L1KarcherUniqueness2}
\lim_{R\to\infty}\int_{\mathbb{P}\setminus B(X,R)}d_\infty(X,A)d\mu(A)=0.
\end{equation}
For $R\in[0,\infty)$, if $\mu(\mathbb{P}\setminus B(X,R))>0$ define
\begin{equation*}
E(R):=\frac{1}{\mu(\mathbb{P}\setminus B(X,R))}\int_{\mathbb{P}\setminus B(X,R)}\log(X^{-1/2}AX^{-1/2})d\mu(A)
\end{equation*}
and $E(R):=0$ otherwise. Also define $Z(R):=X^{1/2}\exp(E(R))X^{1/2}$ and $\mu_R\in\mathcal{P}^1(\mathbb{P})$ by
\begin{equation*}
\mu_R:=\mu|_{B(X,R)}+\mu(\mathbb{P}\setminus B(X,R))\delta_{Z(R)}
\end{equation*}
where $\mu|_{B(X,R)}$ is the restriction of $\mu$ to $B(X,R)$. Note that $\mu_R$ has bounded support for any $R\in(0,\infty)$.

Next, we claim that $\lim_{R\to\infty}W_1(\mu_R,\mu)=0$. If $W_1(\mu_{R_0},\mu)=0$ for some $R_0>0$ then $W_1(\mu_R,\mu)=0$ for all $R\geq R_0$ and we are done, so assume $W_1(\mu_{R},\mu)\neq 0$. We have
\begin{equation*}
\begin{split}
W_1(\mu_R,\mu)&=W_1\left(\mu|_{B(X,R)}+\mu(\mathbb{P}\setminus B(X,R))\delta_{Z(R)},\right.\\
&\left.\quad\quad\quad\mu|_{B(X,R)}+\mu(\mathbb{P}\setminus B(X,R))\frac{1}{\mu(\mathbb{P}\setminus B(X,R))}\mu|_{\mathbb{P}\setminus B(X,R)}\right)\\
&\leq\mu(B(X,R))W_1\left(\frac{1}{\mu(B(X,R))}\mu|_{B(X,R)},\frac{1}{\mu(B(X,R))}\mu|_{B(X,R)}\right)\\
&\quad+\mu(\mathbb{P}\setminus B(X,R))W_1\left(\delta_{Z(R)},\frac{1}{\mu(\mathbb{P}\setminus B(X,R))}\mu|_{\mathbb{P}\setminus B(X,R)}\right)\\
&=\int_{\mathbb{P}\setminus B(X,R)}d_\infty(Z(R),A)d\mu(A)\\
&\leq\int_{\mathbb{P}\setminus B(X,R)}d_\infty(Z(R),X)+d_\infty(X,A)d\mu(A)\\
&=\int_{\mathbb{P}\setminus B(X,R)}\|E(R)\|d\mu(A)+\int_{\mathbb{P}\setminus B(X,R)}d_\infty(X,A)d\mu(A)\\
&=\left\|\int_{\mathbb{P}\setminus B(X,R)}\log(X^{-1/2}AX^{-1/2})d\mu(A)\right\|+\int_{\mathbb{P}\setminus B(X,R)}d_\infty(X,A)d\mu(A)\\
&\leq\int_{\mathbb{P}\setminus B(X,R)}\left\|\log(X^{-1/2}AX^{-1/2})\right\|d\mu(A)+\int_{\mathbb{P}\setminus B(X,R)}d_\infty(X,A)d\mu(A)\\
&=2\int_{\mathbb{P}\setminus B(X,R)}d_\infty(X,A)d\mu(A)
\end{split}
\end{equation*}
where to obtain the first inequality we used \eqref{eq:P:W_1convex}. This proves our claim by \eqref{eq0:T:L1KarcherUniqueness2}.

On one hand, since $\mu_R$ has bounded support for all $R\in(0,\infty)$ by Theorem~\ref{T:KarcherExist2} it follows that the Karcher equation
\begin{equation}\label{eq1:T:L1KarcherUniqueness2}
\int_{\mathbb{P}}\log_YAd\mu_R(A)=0
\end{equation}
has a unique solution in $\mathbb{P}$ and that must be $\Lambda(\mu_R)$ by Theorem~\ref{T:LambdaExists}. On the other hand, we have that by definition $X$ is also a solution of \eqref{eq1:T:L1KarcherUniqueness2}, thus $\Lambda(\mu_R)=X$ for all $R\in(0,\infty)$. Now by Proposition~\ref{P:weakW1agree} we choose a sequence of finitely supported probability measures $\mu_n\in\mathcal{P}^1(\mathbb{P})$ that is $W_1$-converging to $\mu$ so by Theorem~\ref{T:LambdaExists} $\Lambda(\mu_n)\to\Lambda(\mu)$. Then, by the claim $W_1(\mu_R,\mu_n)\to 0$ as $R,n\to \infty$, thus by the contraction property \eqref{eq:T:LambdaExists} $d_\infty(\Lambda(\mu_R),\Lambda(\mu_n))\to 0$, that is $d_\infty(X,\Lambda(\mu_n))\to 0$ and also $\Lambda(\mu_n)\to\Lambda(\mu)$ proving that $X=\Lambda(\mu)$, thus the uniqueness of the solution of \eqref{eq:D:Karcher}.
\end{proof}

\begin{remark}
Many properties of $\Lambda$ now carries over to the $L^1$-setting. The interested reader can consult section 6 in \cite{palfia2} and \cite{lawson}. In particular the stochastic order introduced in \cite{kimlee,lawson} extends the usual element-wise order of uniformly finitely supported measures by introducing upper sets: $U\subseteq\mathbb{P}$ is upper if for an $X\in\mathbb{P}$ there exists an $Y\in U$ such that $Y\leq X$, then $X\in U$. Then the \emph{stochastic order} for $\mu,\nu\in\mathcal{P}^{1}(\mathbb{P})$ is defined as $\mu\leq\nu$ if $\mu(U)\leq\nu(U)$ for all upper sets $U\subseteq\mathbb{P}$. Then the results in \cite{lawson} applies and if $\mu\leq\nu$ then $\Lambda(\mu)\leq \Lambda(\nu)$. This can also be proved by applying the results of section 6 in \cite{kimlee} to the infinite dimensional setting with the monotonicity results of \cite{palfia2} for measures with bounded support.
\end{remark}

\section{An ODE flow of $\Lambda$}
The fundamental $W_1$-contraction property \eqref{eq:T:LambdaExists} enables us to develop an ODE flow theory for $\Lambda$ that resembles the gradient flow theory of its potential function in the finite dimensional $\mathrm{CAT}(0)$-space case, see \cite{limpalfia2,ohta} and the monograph \cite{bacak}. Given a $\mathrm{CAT}(\kappa)$-space $(X,d)$, the \emph{Moreau-Yoshida} resolvent of a lower semicontinuous function $f$ is defined as
\begin{equation*}
J_{\lambda}(x):=\argmin_{y\in X}f(y)+\frac{1}{2\lambda}d^2(x,y)
\end{equation*}
for $\lambda>0$. Then the gradient flow $S(t)$ semigroup of $f$ is defined as
\begin{equation*}
S(t)x_0:=\lim_{n\to\infty}(J_{t/n})^nx_0
\end{equation*}
for $t\in[0,\infty)$ and starting point $x_0\in X$, see \cite{bacak}. However in the infinite dimensional case substituting $d_\infty$ in place of $d$ in the above formulas leads to many difficulties, in particular $d_\infty^2$ is not uniformly convex, moreover $d_\infty$ is not differentiable, since the operator norm $\|\cdot\|$ is an $L^{\infty}$-type norm, hence not smooth. Also the potential function $f$ is not known to exist in the infinite dimensional case of $\mathbb{P}$, since there exists no finite trace on $\mathcal{B}(\mathcal{H})$ to be used to define any Riemannian metric on $\mathbb{P}$. However if we use the formulation of the critical point gradient equation corresponding to the definition of $J_\lambda$ above, we can obtain a reasonable ODE theory in our setting for $\Lambda$.

\begin{definition}[Resolvent operator]
Given $\mu\in\mathcal{P}^1(\mathbb{P})$ we define the resolvent operator for $\lambda>0$ and $X\in\mathbb{P}$ as
\begin{equation}\label{eq:D:resolvent}
J_{\lambda}^{\mu}(X):=\Lambda\left(\frac{\lambda}{\lambda+1}\mu+\frac{1}{\lambda+1}\delta_X\right),
\end{equation}
a solution we obtained in Theorem~\ref{T:LambdaExists} of the
Karcher equation
\begin{equation*}
\frac{\lambda}{\lambda+1}\int_{\mathbb{P}}\log_{Z}Ad\mu(A)+\frac{1}{\lambda+1}\log_{Z}(X)=0
\end{equation*}
for $Z\in\mathbb{P}$ according to Definition~\ref{D:KarcherMean}.
\end{definition}

The resolvent operator exists for $\lambda\in
[0,\infty]$ and provides a continuous path from $X$ to
$\Lambda(\mu)$. An alternative such operator is
$$\Lambda(X\#_{t}\mu),  \ \ \ \ t\in [0,1]$$ where
$$ X\#_{t} \mu = f_{*}(\mu),  \ \ f(A):=X\#_{t}A.$$

We readily obtain the following fundamental contraction property of the resolvent.

\begin{proposition}[Resolvent contraction]\label{P:ResolventContraction}
Given $\mu\in\mathcal{P}^1(\mathbb{P})$, for $\lambda>0$ and $X,Y\in\mathbb{P}$ we have
\begin{equation}\label{eq:P:ResolventContraction}
d_\infty(J_{\lambda}^\mu(X),J_{\lambda}^\mu(Y))\leq \frac{1}{1+\lambda}d_\infty(X,Y).
\end{equation}
\end{proposition}
\begin{proof}
Let $\mu_\alpha\in\mathcal{P}^1(\mathbb{P})$ be a net of finitely supported measures $W_1$-converging to $\mu$ by Proposition~\ref{P:weakW1agree}. Then by the triangle inequality and Proposition~\ref{P:KarcherW1contracts} we get
\begin{equation*}
\begin{split}
&d_\infty(J_{\lambda}^\mu(X),J_{\lambda}^\mu(Y))\\
&\leq d_\infty(J_{\lambda}^\mu(X),J_{\lambda}^{\mu_\alpha}(X))+d_\infty(J_{\lambda}^{\mu_\alpha}(X),J_{\lambda}^{\mu_\alpha}(Y))+d_\infty(J_{\lambda}^{\mu_\alpha}(Y),J_{\lambda}^\mu(Y))\\
&\leq d_\infty(J_{\lambda}^\mu(X),J_{\lambda}^{\mu_\alpha}(X))+\frac{1}{1+\lambda}d_\infty(X,Y)+d_\infty(J_{\lambda}^{\mu_\alpha}(Y),J_{\lambda}^\mu(Y)).
\end{split}
\end{equation*}
Since $d_\infty(J_{\lambda}^\mu(Z),J_{\lambda}^{\mu_\alpha}(Z))\to 0$ as $\alpha\to\infty$ by \eqref{eq:T:LambdaExists}, taking the limit $\alpha\to\infty$ in the above chain of inequalities yields the assertion.
\end{proof}

\begin{proposition}[Resolvent identity]\label{P:ResolventIdentity}
Given $\mu\in\mathcal{P}^1(\mathbb{P})$, for $\tau>\lambda>0$ and $X\in\mathbb{P}$ we have
\begin{equation}\label{eq:P:ResolventIdentity}
J_{\tau}^\mu(X)=J_{\lambda}^\mu\left(J_{\tau}^\mu(X)\#_{\frac{\lambda}{\tau}}X\right).
\end{equation}
\end{proposition}
\begin{proof}
First suppose that $\mu=\sum_{i=1}^nw_i\delta_{A_i}$ where $A_i\in\mathbb{P}$ for $1\leq i\leq n$ and $\omega=(w_1,\ldots,w_n)$ a probability vector. By \eqref{eq:D:resolvent} we have
\begin{equation*}
\tau\int_{\mathbb{P}}\log_{J_{\tau}^\mu(X)}Ad\mu(A)+\log_{J_{\tau}^\mu(X)}X=0
\end{equation*}
and from that it follows that
\begin{equation*}
\begin{split}
\lambda\int_{\mathbb{P}}\log_{J_{\tau}^\mu(X)}Ad\mu(A)+\frac{\lambda}{\tau}\log_{J_{\tau}^\mu(X)}X&=0,\\
\lambda\int_{\mathbb{P}}\log_{J_{\tau}^\mu(X)}Ad\mu(A)+\log_{J_{\tau}^\mu(X)}\left(J_{\tau}^\mu(X)\#_{\frac{\lambda}{\tau}}X\right)&=0,
\end{split}
\end{equation*}
and the above equation still uniquely determines $J_{\tau}^\mu(X)$ as its only positive solution by Theorem~\ref{T:KarcherExist}, thus establishing \eqref{eq:P:ResolventIdentity} for finitely supported measures $\mu$.

The general $\mu\in\mathcal{P}^1(\mathbb{P})$ case of \eqref{eq:P:ResolventIdentity} is obtained by approximating $\mu$ in $W_1$ by a net of finitely supported measures $\mu_{\alpha}\in\mathcal{P}^1(\mathbb{P})$ and using \eqref{eq:T:LambdaExists} to show that $J_{\lambda}^{\mu_\alpha}(X)\to J_{\lambda}^\mu(X)$ in $d_\infty$ and also the fact that $\#_t$ appearing in \eqref{eq:P:ResolventIdentity} is also $d_\infty$-continuous, hence obtaining \eqref{eq:P:ResolventIdentity} in the limit as $\mu_\alpha\to\mu$ in $W_1$.
\end{proof}

\begin{proposition}\label{P:ResolventBound}
Given $\mu\in\mathcal{P}^1(\mathbb{P})$, $\lambda>0$ and $X\in\mathbb{P}$ we have
\begin{equation}\label{eq:P:ResolventBound}
\begin{split}
d_\infty(J_{\lambda}^\mu(X),X)&\leq \frac{\lambda}{1+\lambda}\int_{\mathbb{P}}d_\infty(X,A)d\mu(A)\\
d_\infty\left(\left(J_{\lambda}^\mu\right)^n(X),X\right)&\leq n\frac{\lambda}{1+\lambda}\int_{\mathbb{P}}d_\infty(X,A)d\mu(A).
\end{split}
\end{equation}
\end{proposition}
\begin{proof}
By Theorem~\ref{T:LambdaExists} $J_{\lambda}^\mu(X)$ is a solution of
\begin{equation}\label{eq1:P:ResolventBound}
\lambda\int_{\mathbb{P}}\log_{J_{\lambda}^\mu(X)}Ad\mu(A)+\log_{J_{\lambda}^\mu(X)}X=0,
\end{equation}
hence we have
\begin{equation*}
\begin{split}
d_\infty(J_{\lambda}^\mu(X),X)&=\left\|\log\left(J_{\lambda}^\mu(X)^{-1/2}XJ_{\lambda}^\mu(X)^{-1/2}\right)\right\|\\
&=\lambda\left\|\int_{\mathbb{P}}\log\left(J_{\lambda}^\mu(X)^{-1/2}AJ_{\lambda}^\mu(X)^{-1/2}\right)d\mu(A)\right\|\\
&\leq \lambda\int_{\mathbb{P}}\left\|\log\left(J_{\lambda}^\mu(X)^{-1/2}AJ_{\lambda}^\mu(X)^{-1/2}\right)\right\|d\mu(A)\\
&=\lambda\int_{\mathbb{P}}d_\infty(J_{\lambda}^\mu(X),A)d\mu(A)
\end{split}
\end{equation*}
Given $J_{\lambda}^\mu(X)\in\mathbb{P}$ we can solve \eqref{eq1:P:ResolventBound} for $X\in\mathbb{P}$, thus by Proposition~\ref{P:ResolventContraction} we also have
\begin{equation*}
d_\infty(J_{\tau}^\mu(X),X)=d_\infty\left(J_{\tau}^\mu(X),J_{\tau}^\mu\left(\left(J_{\tau}^\mu\right)^{-1}(X)\right)\right)\leq \frac{1}{1+\lambda}d_\infty\left(X,\left(J_{\tau}^\mu\right)^{-1}(X)\right),
\end{equation*}
hence the first inequality in \eqref{eq:P:ResolventBound} follows.

The second inequality in \eqref{eq:P:ResolventBound} follows from the first by the estimate
\begin{equation*}
\begin{split}
d_\infty\left(\left(J_{\lambda}^\mu\right)^n(X),X\right)&\leq \sum_{i=0}^{n-1}d_\infty\left(\left(J_{\lambda}^\mu\right)^{n-i}(X),\left(J_{\lambda}^\mu\right)^{n-(i+1)}(X)\right)\\
&\leq \sum_{i=0}^{n-1}(1+\lambda)^{-n+(i+1)}d_\infty\left(J_{\lambda}^\mu(X),X\right)\\
&\leq nd_\infty\left(J_{\lambda}^\mu(X),X\right).
\end{split}
\end{equation*}
\end{proof}

In what follows we will closely follow the arguments in \cite{crandall} to construct the semigroups corresponding to the resolvent above. $B(k,l)$ denotes the binomial coefficient.

\begin{lemma}[a variant of Lemma 1.3 cf. \cite{crandall}]\label{L:Crandall}
Let $\mu\in\mathcal{P}^1(\mathbb{P})$, $\tau\geq\lambda>0$; $n\geq m$ be positive integers and $X\in\mathbb{P}$. Then
\begin{equation*}
\begin{split}
d_\infty&\left(\left(J_{\tau}^\mu\right)^n(X),\left(J_{\lambda}^\mu\right)^m(X)\right)\\
&\leq (1+\lambda)^{-n}\sum_{j=0}^{m-1}\alpha^j\beta^{n-j}B(n,j)d_\infty\left(\left(J_{\tau}^\mu\right)^{m-j}(X),X\right)\\
&\quad+\sum_{j=m}^{n}(1+\lambda)^{-j}\alpha^m\beta^{j-m}B(j-1,m-1)d_\infty\left(\left(J_{\lambda}^\mu\right)^{n-j}(X),X\right)
\end{split}
\end{equation*}
where $\alpha=\frac{\lambda}{\tau}$ and $\beta=\frac{\tau-\lambda}{\tau}$.
\end{lemma}
\begin{proof}
For integers $j$ and $k$ satisfying $0\leq j\leq n$ and $0\leq k\leq m$, put
\begin{equation*}
a_{k,j}:=d_\infty\left(\left(J_{\lambda}^\mu\right)^{j}(X),\left(J_{\tau}^\mu\right)^{k}(X)\right).
\end{equation*}
For $j,k>0$ by Proposition~\ref{P:ResolventContraction} and Proposition~\ref{P:ResolventIdentity} we have
\begin{equation*}
\begin{split}
a_{k,j}&=d_\infty\left(\left(J_{\lambda}^\mu\right)^{j}(X),J_{\lambda}^\mu\left(\left(J_{\tau}^\mu\right)^{k}(X)\#_{\frac{\lambda}{\tau}}\left(J_{\tau}^\mu\right)^{k-1}(X)\right)\right)\\
&\leq (1+\lambda)^{-1}d_\infty\left(\left(J_{\lambda}^\mu\right)^{j-1}(X),\left(J_{\tau}^\mu\right)^{k}(X)\#_{\frac{\lambda}{\tau}}\left(J_{\tau}^\mu\right)^{k-1}(X)\right)\\
&\leq (1+\lambda)^{-1}\left[\frac{\tau-\lambda}{\tau}d_\infty\left(\left(J_{\lambda}^\mu\right)^{j-1}(X),\left(J_{\tau}^\mu\right)^{k}(X)\right)\right.\\
&\quad\left.+\frac{\lambda}{\tau}d_\infty\left(\left(J_{\lambda}^\mu\right)^{j-1}(X),\left(J_{\tau}^\mu\right)^{k-1}(X)\right)\right]\\
&=(1+\lambda)^{-1}\frac{\lambda}{\tau}a_{k-1,j-1}+(1+\lambda)^{-1}\frac{\tau-\lambda}{\tau}a_{k,j-1},
\end{split}
\end{equation*}
where to obtain the second inequality we used Proposition~\ref{P:KarcherW1contracts} for $\#_t$. From here, the rest of the proof follows along the lines of Lemma 1.3 in \cite{crandall}.
\end{proof}

We quote the following Lemma 1.4. from \cite{crandall}:

\begin{lemma}
Let $n\geq m>0$ be integers, and $\alpha,\beta$ positive numbers satisfying $\alpha+\beta=1$. Then
\begin{equation*}
\sum_{j=0}^{m}B(n,j)\alpha^j\beta^{n-j}(m-j)\leq \sqrt{(n\alpha-m)^2+n\alpha\beta},
\end{equation*}
and
\begin{equation*}
\sum_{j=m}^{n}B(j-1,m-1)\alpha^m\beta^{j-m}(n-j)\leq \sqrt{\frac{m\beta}{\alpha^2}+\left(\frac{m\beta}{\alpha^2}+m-n\right)^2}.
\end{equation*}
\end{lemma}

\begin{theorem}\label{T:ExponentialFormula}
For any $X,Y\in\mathbb{P}$ and $t>0$ the curve
\begin{equation}\label{eq1:T:ExponentialFormula}
S(t)X:=\lim_{n\to\infty}\left(J_{t/n}^{\mu}\right)^{n}(X)
\end{equation}
exists where the limit is in the $d_\infty$-topology and it is Lipschitz-continuous on compact time intervals $[0,T]$ for any $T>0$. Moreover it satisfies the contraction property
\begin{equation}\label{eq2:T:ExponentialFormula}
d_\infty\left(S(t)X,S(t)Y\right)\leq e^{-t}d_\infty(X,Y),
\end{equation}
and for $s>0$ verifies the semigroup property
\begin{equation}\label{eq3:T:ExponentialFormula}
S(t+s)X=S(t)(S(s)X),
\end{equation}
and the flow operator $S:\mathbb{P}\times(0,\infty)\mapsto \mathbb{P}$ extends by $d_\infty$-continuity to $S:\mathbb{P}\times[0,\infty)\mapsto \mathbb{P}$.
\end{theorem}
\begin{proof}
The proof closely follows that of Theorem I in \cite{crandall} using the previous estimates of this section. In particular for $n\geq m>0$ one obtains
\begin{equation}\label{eq4:T:ExponentialFormula}
d_\infty\left(\left(J_{t/n}^\mu\right)^{n}(X),\left(J_{t/m}^\mu\right)^{m}(X)\right)\leq 2t\left(\frac{1}{m}-\frac{1}{n}\right)^{1/2}\int_{\mathbb{P}}d_\infty(X,A)d\mu(A)
\end{equation}
so $\lim_{n\to\infty}\left(J_{t/n}^\mu\right)^{n}(X)$ exists proving \eqref{eq1:T:ExponentialFormula}. Also $\left(J_{t/n}^\mu\right)^{n}$ satisfies
\begin{equation*}
d_\infty\left(\left(J_{t/n}^\mu\right)^{n}(X),\left(J_{t/n}^\mu\right)^{n}(Y)\right)\leq \left(1+\frac{t}{n}\right)^{-n}d_\infty(X,Y),
\end{equation*}
hence also \eqref{eq2:T:ExponentialFormula}. We also have
\begin{equation}\label{eq5:T:ExponentialFormula}
d_\infty\left(S(s)X,S(t)X\right)\leq 2|s-t|\int_{\mathbb{P}}d_\infty(X,A)d\mu(A)
\end{equation}
proving Lipschitz-continuity in $t$ on compact time intervals. The proof of the semigroup property is exactly the same as in \cite{crandall}.
\end{proof}

Before stating the next result we need another auxiliary lemma describing the asymptotic behavior of $J_{t/n}^{\mu}(X)$.

\begin{lemma}\label{L:ResolventAsymptotics}
Let $\mu\in\mathcal{P}^1(\mathbb{P})$, $\lambda>0$ and $X\in\mathbb{P}$. Then
\begin{equation}\label{eq:L:ResolventAsymptotics}
\log_{J_{\lambda}^{\mu}(X)}X=X-J_{\lambda}^{\mu}(X)+O\left(\lambda^2\right).
\end{equation}
\end{lemma}
\begin{proof}
Let $C:=\int_{\mathbb{P}}d_\infty(X,A)d\mu(A)$. Then by Proposition~\ref{P:ResolventBound}
\begin{equation*}
e^{-\lambda\left(1+\lambda\right)^{-1}C}-I\leq J_{\lambda}^{\mu}(X)^{-1/2}XJ_{\lambda}^{\mu}(X)^{-1/2}-I\leq e^{\lambda\left(1+\lambda\right)^{-1}C}-I,
\end{equation*}
hence
\begin{equation*}
e^{-\lambda C}-I\leq J_{\lambda}^{\mu}(X)^{-1/2}XJ_{\lambda}^{\mu}(X)^{-1/2}-I\leq e^{\lambda C}-I,
\end{equation*}
which yields
\begin{equation}\label{eq1:L:ResolventAsymptotics}
\sum_{k=1}^\infty (-1)^k\frac{\left(\lambda C\right)^k}{k!}\leq J_{\lambda}^{\mu}(X)^{-1/2}XJ_{\lambda}^{\mu}(X)^{-1/2}-I\leq \sum_{k=1}^\infty \frac{\left(\lambda C\right)^k}{k!}.
\end{equation}
In view of the series expansion
\begin{equation*}
\log(z)=\sum_{k=1}^\infty(-1)^{k-1}\frac{(z-I)^k}{k}
\end{equation*}
uniformly convergent for $\|z-I\|<1$, we get
\begin{equation*}
\begin{split}
\log\left(J_{\lambda}^{\mu}(X)^{-1/2}XJ_{\lambda}^{\mu}(X)^{-1/2}\right)&\\
=J_{\lambda}^{\mu}(X)^{-1/2}&XJ_{\lambda}^{\mu}(X)^{-1/2}-I+O\left(\lambda^2\right),
\end{split}
\end{equation*}
from which the assertion follows.
\end{proof}

The proof of the following theorem in essence is analogous to that of Theorem II in \cite{crandall}.

\begin{theorem}\label{T:StrongSolution}
Let $\mu\in\mathcal{P}^1(\mathbb{P})$ and $X\in\mathbb{P}$. Then for $t>0$, the curve $X(t):=S(t)X$ provides a strong solution of the Cauchy problem
\begin{equation*}
\begin{split}
X(0)&:=X,\\
\dot{X}(t)&=\int_{\mathbb{P}}\log_{X(t)}Ad\mu(A),
\end{split}
\end{equation*}
where the derivative $\dot{X}(t)$ is the Fr\'echet-derivative.
\end{theorem}
\begin{proof}
Due to the semigroup property of $S(t)$, it is enough to check that
\begin{equation*}
\lim_{t\to 0+}\frac{S(t)X-X}{t}=\int_{\mathbb{P}}\log_{X}Ad\mu(A)
\end{equation*}
where the limit is in the norm topology. We have
\begin{equation*}
\begin{split}
\frac{S(t)X-X}{t}&=\lim_{n\to\infty}\frac{\left(J_{t/n}^{\mu}\right)^n(X)-X}{t}\\
&=\lim_{n\to\infty}\frac{1}{n}\frac{\sum_{i=0}^{n-1}J_{t/n}^{\mu}\left(\left(J_{t/n}^{\mu}\right)^i(X)\right)-\left(J_{t/n}^{\mu}\right)^i(X)}{t/n}
\end{split}
\end{equation*}
and also
\begin{equation*}
\frac{t}{n}\int_{\mathbb{P}}\log_{\left(J_{t/n}^{\mu}\right)^i(X)}Ad\mu(A)+\log_{\left(J_{t/n}^{\mu}\right)^i(X)}\left(J_{t/n}^{\mu}\right)^{i-1}(X)=0.
\end{equation*}
Then by Lemma~\ref{L:ResolventAsymptotics} we have
\begin{equation*}
\frac{S(t)X-X}{t}=\lim_{n\to\infty}\frac{1}{n}\sum_{i=0}^{n-1}\int_{\mathbb{P}}\log_{\left(J_{t/n}^{\mu}\right)^i(X)}Ad\mu(A)+O\left(\frac{t}{n}\right),
\end{equation*}
which combined with the estimates in Proposition~\ref{P:ResolventBound} and Lemma~\ref{L:gradContinuous} proves the assertion.
\end{proof}

\begin{proposition}\label{P:StationaryFlow}
Let $\mu\in\mathcal{P}^1(\mathbb{P})$. Then the semigroup $S(t)\Lambda(\mu)$ generated in Theorem~\ref{T:ExponentialFormula} is stationary, that is $S(t)\Lambda(\mu)=\Lambda(\mu)$ for all $t>0$.
\end{proposition}
\begin{proof}
It is enough to show that $J^\mu_\lambda(\Lambda(\mu))=\Lambda(\mu)$ for any $\lambda>0$. Indeed by substitution $\Lambda(\mu)$ is a solution of
\begin{equation*}
\frac{\lambda}{\lambda+1}\int_{\mathbb{P}}\log_{Z}Ad\mu(A)+\frac{1}{\lambda+1}\log_{Z}(\Lambda(\mu))=0
\end{equation*}
but this solution is unique by Theorem~\ref{T:L1KarcherUniqueness} and by definition \eqref{eq:D:resolvent} it is $J^\mu_\lambda(\Lambda(\mu))$.
\end{proof}

\begin{problem}
Are the solution curves $\gamma:[0,\infty)\mapsto \mathbb{P}$ of the Cauchy problem in Theorem~\ref{T:StrongSolution} unique?
\end{problem}

\section{Approximating semigroups and Trotter-Kato product formula}
In this section we develop the theory of approximating semigroups that will lead to a Trotter-Kato product formula for the nonlinear ODE semigroups of the Karcher mean.

\begin{lemma}\label{L:ApproxResolvent}
Let $F:\mathbb{P}\mapsto\mathbb{P}$ be a nonexpansive map with respect to $d_\infty$. Let $\lambda,\rho>0$ and $Y\in\mathbb{P}$. Then the map
\begin{equation*}
G_{\lambda,\rho,Y}(X):=\Lambda\left(\frac{1}{1+\lambda/\rho}\delta_{Y}+\frac{\lambda/\rho}{1+\lambda/\rho}\delta_{F(X)}\right)
\end{equation*}
is a strict contraction with Lipschitz constant $\frac{\lambda/\rho}{1+\lambda/\rho}<1$. Consequently the map $G_{\lambda,\rho,Y}$ has a unique fixed point denoted by $J_{\lambda,\rho}(Y)$.
\end{lemma}
\begin{proof}
By Proposition~\ref{P:KarcherW1contracts} for $X_1,X_2\in\mathbb{P}$ we get
\begin{equation*}
d_\infty(G_{\lambda,\rho,Y}(X_1),G_{\lambda,\rho,Y}(X_2))\leq \frac{\lambda/\rho}{1+\lambda/\rho}d_\infty(F(X_1),F(X_2))\leq \frac{\lambda/\rho}{1+\lambda/\rho}d_\infty(X_1,X_2),
\end{equation*}
thus by Banach's fixed point theorem $G_{\lambda,\rho,Y}$ has a unique fixed point denoted by $J_{\lambda,\rho}(Y)$, hence
\begin{equation}\label{eq:L:ApproxResolvent1}
J_{\lambda,\rho}(Y)=\Lambda\left(\frac{1}{1+\lambda/\rho}\delta_{Y}+\frac{\lambda/\rho}{1+\lambda/\rho}\delta_{F(J_{\lambda,\rho}(Y))}\right).
\end{equation}
\end{proof}

\begin{lemma}\label{L:ApproxResolventContr}
Let $F:\mathbb{P}\mapsto\mathbb{P}$ be a nonexpansive map with respect to $d_\infty$. Then for $\lambda,\rho>0$ the map $J_{\lambda,\rho}:\mathbb{P}\mapsto\mathbb{P}$ is nonexpansive.
\end{lemma}
\begin{proof}
By Proposition~\ref{P:KarcherW1contracts} and Lemma~\ref{L:ApproxResolvent} for $X_1,X_2\in\mathbb{P}$ and $t:=\frac{\lambda/\rho}{1+\lambda/\rho}<1$ we get
\begin{equation*}
\begin{split}
d_\infty(J_{\lambda,\rho}(X_1),J_{\lambda,\rho}(X_2))&\leq (1-t)d_\infty(X_1,X_2)+td_\infty(F(J_{\lambda,\rho}(X_1)),F(J_{\lambda,\rho}(X_2)))\\
&\leq (1-t)d_\infty(X_1,X_2)+td_\infty(J_{\lambda,\rho}(X_1),J_{\lambda,\rho}(X_2)),
\end{split}
\end{equation*}
from which $d_\infty(J_{\lambda,\rho}(X_1),J_{\lambda,\rho}(X_2))\leq d_\infty(X_1,X_2)$ follows.
\end{proof}

\begin{lemma}\label{L:ApproxResolventEst}
Let $F:\mathbb{P}\mapsto\mathbb{P}$ be a nonexpansive map. Then for $\lambda,\rho>0$ and $X\in\mathbb{P}$ we have
\begin{equation}\label{eq:L:ApproxResolventEst}
\frac{d_\infty(X,J_{\lambda,\rho}(X))}{\lambda}\leq \frac{d_\infty(X,F(X))}{\rho}.
\end{equation}
\end{lemma}
\begin{proof}
By Lemma~\ref{L:ApproxResolvent} we have $\lim_{n\to\infty}G_{\lambda,\rho,X}^n(X)=J_{\lambda,\rho}(X)$. Then
\begin{equation*}
\begin{split}
d_\infty(X,J_{\lambda,\rho}(X))&\leq \sum_{n=1}^{\infty}d_\infty(G_{\lambda,\rho,X}^{n-1}(X),G_{\lambda,\rho,X}^n(X))\\
&\leq \sum_{n=1}^{\infty}\left(\frac{\lambda/\rho}{1+\lambda/\rho}\right)^{n-1}d_\infty(X,G_{\lambda,\rho,X}(X))\\
&\leq \frac{1}{1-\frac{\lambda/\rho}{1+\lambda/\rho}}d_\infty(X,G_{\lambda,\rho,X}(X))=\frac{\lambda}{\rho}d_\infty(X,F(X)),
\end{split}
\end{equation*}
since $d_\infty(X,G_{\lambda,\rho,X}(X))=\frac{\lambda/\rho}{1+\lambda/\rho}d_\infty(X,F(X))$.
\end{proof}

\begin{lemma}[Resolvent Identity]\label{L:ApproxResolventIdentity}
Let $F:\mathbb{P}\mapsto\mathbb{P}$ be a nonexpansive map. Then for $\lambda>\mu,\rho>0$ and $X\in\mathbb{P}$ we have
\begin{equation}\label{eq:L:ApproxResolventIdentity}
J_{\lambda,\rho}(X)=J_{\mu,\rho}\left(J_{\lambda,\rho}(X)\#_{\frac{\mu}{\lambda}}X\right).
\end{equation}
\end{lemma}
\begin{proof}
First of all notice that for $A,B\in\mathbb{P}$ we have that the curve $c(t)=A\#_tB$ for $t\in[0,1]$ has the property that for any $s\leq u\in[0,1]$ the curve $v(t):=c(s)\#_tc(u)$ is a connected subset of the curve $c$, i.e. $v(t)=c(s+t(u-s))$.

Now consider the curve $\gamma(t):=X\#_tF(J_{\lambda,\rho}(X))$ for $t\in[0,1]$. Then by definition it follows that the points $J_{\lambda,\rho}(X)$ and $Z:=J_{\lambda,\rho}(X)\#_{\frac{\mu}{\lambda}}X$ are also on the curve $\gamma$, hence by the above $c(t):=Z\#_tF(J_{\lambda,\rho}(X))$ for $t\in[0,1]$ is a connected subset of the curve $\gamma$. Then to conclude our assertion, by \eqref{eq:L:ApproxResolvent1} and Theorem~\ref{T:KarcherExist}, it suffices to show that
\begin{equation*}
\frac{d_\infty(Z,J_{\lambda,\rho}(X))}{d_\infty(Z,F(J_{\lambda,\rho}(X)))}=\frac{\mu}{\rho+\mu}.
\end{equation*}

Indeed, let $a:=d_\infty(X,F(J_{\lambda,\rho}(X)))$, $b:=d_\infty(J_{\lambda,\rho}(X),F(J_{\lambda,\rho}(X)))$ and $d:=d_\infty(Z,J_{\lambda,\rho}(X))$. Then $b=\frac{\rho}{\rho+\lambda}a$, $d=(a-b)\frac{\mu}{\lambda}=\frac{\lambda}{\rho+\lambda}\frac{\mu}{\lambda}a=\frac{\mu}{\rho+\lambda}a$, thus we have
\begin{equation*}
\frac{d_\infty(Z,J_{\lambda,\rho}(X))}{d_\infty(Z,F(J_{\lambda,\rho}(X)))}=\frac{d}{d+b}=\frac{\mu}{\rho+\mu}.
\end{equation*}
\end{proof}

\begin{lemma}\label{L:ApproxResolventEst2}
Let $F:\mathbb{P}\mapsto\mathbb{P}$ be a nonexpansive map. Then for $\lambda,\rho>0,n\in\mathbb{N}$ and $X\in\mathbb{P}$ we have
\begin{equation}\label{eq:L:ApproxResolventEst2}
d_\infty(J_{\lambda,\rho}^n(X),X)\leq n\frac{\lambda}{\rho}d_\infty(X,F(X)).
\end{equation}
\end{lemma}
\begin{proof}
By the triangle inequality, Lemma~\ref{L:ApproxResolventContr} and Lemma~\ref{L:ApproxResolventEst} we have
\begin{equation*}
\begin{split}
d_\infty(J_{\lambda,\rho}^n(X),X)&\leq \sum_{i=1}^nd_\infty(J_{\lambda,\rho}^i(X),J_{\lambda,\rho}^{i-1}(X))\\
&\leq nd_\infty(J_{\lambda,\rho}(X),X)\\
&\leq n\frac{\lambda}{\rho}d_\infty(X,F(X)).
\end{split}
\end{equation*}
\end{proof}

\begin{lemma}
Let $F:\mathbb{P}\mapsto\mathbb{P}$ be a nonexpansive map, $\tau\geq\lambda>0$, $\rho>0$; $n\geq m$ be positive integers and $X\in\mathbb{P}$. Then
\begin{equation*}
\begin{split}
d_\infty&\left(\left(J_{\tau,\rho}\right)^n(X),\left(J_{\lambda,\rho}\right)^m(X)\right)\\
&\leq \sum_{j=0}^{m-1}\alpha^j\beta^{n-j}B(n,j)d_\infty\left(\left(J_{\tau,\rho}\right)^{m-j}(X),X\right)\\
&\quad+\sum_{j=m}^{n}\alpha^m\beta^{j-m}B(j-1,m-1)d_\infty\left(\left(J_{\lambda,\rho}\right)^{n-j}(X),X\right)
\end{split}
\end{equation*}
where $\alpha=\frac{\lambda}{\tau}$ and $\beta=\frac{\tau-\lambda}{\tau}$.
\end{lemma}
\begin{proof}
Using the Resolvent Identity Lemma~\ref{L:ApproxResolventIdentity}, the estimates are obtained in the same way as in Lemma~\ref{L:Crandall}.
\end{proof}

\begin{theorem}\label{T:ExponentialFormulaAppr}
Let $F:\mathbb{P}\mapsto\mathbb{P}$ be a nonexpansive map. Then for any $X,Y\in\mathbb{P}$ and $t,\rho>0$ the curve
\begin{equation}\label{eq1:T:ExponentialFormulaAppr}
S_{\rho}(t)X:=\lim_{n\to\infty}\left(J_{t/n,\rho}\right)^{n}(X)
\end{equation}
exists where the limit is in the $d_\infty$-topology with estimate
\begin{equation}\label{eq2:T:ExponentialFormulaAppr}
d_\infty(S_{\rho}(t)X,\left(J_{t/n,\rho}\right)^{n}(X))\leq \frac{2t}{\sqrt{n}}\frac{d_\infty(X,F(X))}{\rho},
\end{equation}
and satisfies the Lipschitz estimate
\begin{equation}\label{eq3:T:ExponentialFormulaAppr}
d_\infty(S_{\rho}(t)X,S_{\rho}(s)X)\leq 2\frac{d_\infty(X,F(X))}{\rho}|t-s|
\end{equation}
for any $t,s\geq 0$. Moreover it also satisfies the contraction property
\begin{equation}\label{eq4:T:ExponentialFormulaAppr}
d_\infty\left(S_{\rho}(t)X,S_{\rho}(t)Y\right)\leq d_\infty(X,Y),
\end{equation}
for $s>0$ verifies the semigroup property
\begin{equation}\label{eq5:T:ExponentialFormulaAppr}
S_{\rho}(t+s)X=S_{\rho}(t)(S_{\rho}(s)X),
\end{equation}
and the flow operator $S_{\rho}:\mathbb{P}\times(0,\infty)\mapsto \mathbb{P}$ extends by $d_\infty$-continuity to $S_{\rho}:\mathbb{P}\times[0,\infty)\mapsto \mathbb{P}$.
\end{theorem}
\begin{proof}
We closely follow the proof of Theorem~\ref{T:ExponentialFormula}. Using the previous lemmas we similarly obtain
\begin{equation}\label{eq:T:ExponentialFormulaAppr1}
\begin{split}
d_\infty&\left(\left(J_{\tau,\rho}\right)^n(X),\left(J_{\lambda,\rho}\right)^m(X)\right)\\
&\leq \sum_{j=0}^{m-1}\alpha^j\beta^{n-j}B(n,j)d_\infty\left(\left(J_{\tau,\rho}\right)^{m-j}(X),X\right)\\
&\quad+\sum_{j=m}^{n}\alpha^m\beta^{j-m}B(j-1,m-1)d_\infty\left(\left(J_{\lambda,\rho}\right)^{n-j}(X),X\right)\\
&\leq \sum_{j=0}^{m-1}\alpha^j\beta^{n-j}B(n,j)(m-j)\frac{d_\infty(X,F(X))}{\rho}\\
&\quad+\sum_{j=m}^{n}\alpha^m\beta^{j-m}B(j-1,m-1)(n-j)\frac{d_\infty(X,F(X))}{\rho}\\
&\leq \left[\lambda\sqrt{\left(n\frac{\tau}{\lambda}-m\right)^2+n\frac{\tau}{\lambda}\frac{\lambda-\tau}{\lambda}}\right.\\
&\quad\left.+\tau\sqrt{\frac{\lambda^2}{\tau^2}\frac{\lambda-\tau}{\lambda}m+\left(\frac{\lambda}{\tau}\frac{\lambda-\tau}{\lambda}m+m-n\right)^2}\right]\frac{d_\infty(X,F(X))}{\rho}\\
&=\left[\sqrt{(n\tau-m\lambda)^2+n\tau(\lambda-\tau)}\right.\\
&\quad\left.+\sqrt{m\lambda(\lambda-\tau)+(m\lambda-n\tau)^2}\right]\frac{d_\infty(X,F(X))}{\rho}.
\end{split}
\end{equation}
For $\tau=\frac{t}{n},\lambda=\frac{t}{m}$, the above reads
\begin{equation*}
d_\infty\left(\left(J_{t/n,\rho}\right)^{n}(X),\left(J_{t/m,\rho}\right)^{m}(X)\right)\leq 2t\left|\frac{1}{n}-\frac{1}{m}\right|^{1/2}\frac{d_\infty(X,F(X))}{\rho},
\end{equation*}
so the limit in \eqref{eq1:T:ExponentialFormulaAppr} exists by completeness and satisfies \eqref{eq2:T:ExponentialFormulaAppr}, moreover the above also yield the Lipschitz estimate \eqref{eq3:T:ExponentialFormulaAppr}. The rest of the properties is routine to prove, by following the steps of the proof of Theorem~\ref{T:ExponentialFormula}.
\end{proof}

\begin{lemma}\label{L:ApprErrorEst}
Let $F:\mathbb{P}\mapsto\mathbb{P}$ be a nonexpansive map, $\rho>0$ and let $S_\rho(t)$ be the semigroup constructed in Theorem~\ref{T:ExponentialFormulaAppr}. Then for $t>0$, $X\in\mathbb{P}$ and $m\in\mathbb{N}$ we have
\begin{itemize}
\item[(i)]$S_\rho(t)X=S_1(t/\rho)X$,
\item[(ii)]$d_\infty(F^m(X),S_\rho(t)X)\leq \left[\frac{t}{\rho}-m+2\sqrt{\left(\frac{t}{\rho}-m\right)^2+\frac{t}{\rho}}\right]d_\infty(X,F(X))$.
\end{itemize}
\end{lemma}
\begin{proof}
The proof of (i) follows directly from the fact that for $\rho,\lambda>0$, we have $J_{\lambda,\rho}=J_{\lambda/\rho,1}$.

We turn to the proof of (ii) which is more involved. Proposition~\ref{P:KarcherW1contracts} and \eqref{eq:L:ApproxResolvent1} yields that
\begin{equation*}
\begin{split}
d_\infty(F^m(X),(J_{t/n,1})^n(X))&\leq \frac{1}{1+\frac{t}{n}}d_\infty(F^m(X),(J_{t/n,1})^{n-1}(X))\\
&\quad+\frac{\frac{t}{n}}{1+\frac{t}{n}}d_\infty(F^m(X),F((J_{t/n,1})^{n}(X))).
\end{split}
\end{equation*}
Using the above inequality recursively, we get
\begin{equation*}
\begin{split}
d_\infty(F^m(X),(J_{t/n,1})^n(X))&\leq \left(1+\frac{t}{n}\right)^{-n}d_\infty(F^m(X),X)\\
&+\frac{t}{n}\sum_{k=1}^n\left(1+\frac{t}{n}\right)^{-(n-k)}d_\infty(F^{m}(X),F((J_{t/n,1})^{k}(X)))\\
&\leq \left(1+\frac{t}{n}\right)^{-n}md_\infty(F(X),X)\\
&+\frac{t}{n}\sum_{k=1}^n\left(1+\frac{t}{n}\right)^{-(n-k)}d_\infty(F^{m-1}(X),(J_{t/n,1})^{k}(X)).
\end{split}
\end{equation*}
For $n\in\mathbb{N}$ define
\begin{equation*}
\begin{split}
f_n(s)&:=\sum_{k=1}^n\left(1+\frac{t}{n}\right)^{-(n-k)}1_{\left(\frac{(k-1)t}{n},\frac{kt}{n}\right]}(s),\\
g_n(s)&:=\sum_{k=1}^nd_\infty(F^{m-1}(X),(J_{t/n,1})^{k}(X))1_{\left(\frac{(k-1)t}{n},\frac{kt}{n}\right]}(s),
\end{split}
\end{equation*}
so that the above becomes
\begin{equation*}
d_\infty(F^m(X),(J_{t/n,1})^n(X))\leq \left(1+\frac{t}{n}\right)^{-n}md_\infty(F(X),X)+\int_{0}^tf_n(s)g_n(s)ds.
\end{equation*}
We will show that $f_n(s)\to e^{-(t-s)}$ and $g_n(s)\to d_\infty(F^{m-1}(X),S_1(s)X)$ for $s\in(0,t]$ and that $\sup_{n\in\mathbb{N},s\in[0,t]}|f_n(s)||g_n(s)|<\infty$. Then by dominated convergence we will have
\begin{equation}\label{eq:L:ApprErrorEst1}
\begin{split}
d_\infty(F^m(X),S_1(t)X)&\leq e^{-t}md_\infty(X,F(X))\\
&\quad+\int_{0}^te^{-(t-s)}d_\infty(F^{m-1}(X),S_1(s)X)ds.
\end{split}
\end{equation}

Firstly it is routine to see that $f_n(s)\to e^{-(t-s)}$. To prove the other claim, let $n\in\mathbb{N}$ and $s\in(0,t]$. There is a unique $0<k_{s,n}\leq n$ such that
\begin{equation*}
\frac{(k_{s,n}-1)t}{n}<s \leq \frac{k_{s,n}t}{n}.
\end{equation*}
Substituting $k_{s,n},\frac{t}{n},m,\frac{s}{m}$ for $n,\tau,m,\lambda$ respectively in \eqref{eq:T:ExponentialFormulaAppr1} gives
\begin{equation*}
\begin{split}
d_\infty((J_{t/n,1})^{k_{s,n}}(X),&(J_{s/m,1})^{m}(X))\leq \left[\sqrt{\left(\frac{k_{s,n}t}{n}-s\right)^2+\frac{k_{s,n}t}{n}\left(\frac{s}{m}-\frac{t}{n}\right)}\right.\\
&\left.+\sqrt{s\left(\frac{s}{m}-\frac{t}{n}\right)+\left(\frac{k_{s,n}t}{n}-s\right)^2}\right]2d_\infty(X,F(X))
\end{split}
\end{equation*}
and taking the limit $m\to \infty$ we get
\begin{equation}\label{eq:L:ApprErrorEst2}
\begin{split}
d_\infty((J_{t/n,1})^{k_{s,n}}(X),&S_1(s)(X))\leq \left[\sqrt{\left(\frac{k_{s,n}t}{n}-s\right)^2-\frac{k_{s,n}t}{n}\frac{t}{n}}\right.\\
&\left.+\sqrt{\left(\frac{k_{s,n}t}{n}-s\right)^2-s\frac{t}{n}}\right]2d_\infty(X,F(X)).
\end{split}
\end{equation}
We also have
\begin{equation*}
\begin{split}
|g_n(s)-d_\infty(F^{m-1}(X),S_1(s)X)|=&|d_\infty(F^{m-1}(X),(J_{t/n,1})^{k_{s,n}}(X))\\
&-d_\infty(F^{m-1}(X),S_1(s)X)|
\end{split}
\end{equation*}
which combined with \eqref{eq:L:ApprErrorEst2} and the triangle inequality yields
\begin{equation*}
g_n(s)\to d_\infty(F^{m-1}(X),S_1(s)X),
\end{equation*}
so \eqref{eq:L:ApprErrorEst1} follows. Now if $d_\infty(F(X),X)=0$, then $J_{\lambda,\rho}(X)=X$ for all $\lambda,\rho>0$, and we have $S_\rho(s)X=X$ and (ii) follows. Assume $d_\infty(F(X),X)>0$ and for $m\geq 0, s\geq 0$ let
\begin{equation*}
\phi_m(s):=\frac{d_\infty(F^m(X),S_1(s)X)}{d_\infty(F(X),X)},
\end{equation*}
so that \eqref{eq:L:ApprErrorEst1} gives
\begin{equation*}
\begin{split}
\phi_m(t)&\leq e^{-t}m+\int_{0}^te^{-(t-s)}\phi_{m-1}(s)ds\\
&=e^{-t}\left(m+\int_{0}^te^{s}\phi_{m-1}(s)ds\right),
\end{split}
\end{equation*}
so that if $f_m(s):=e^s\phi_m(s)$, then
\begin{equation}\label{eq:L:ApprErrorEst3}
f_m(t)\leq m+\int_{0}^tf_{m-1}(s)ds.
\end{equation}
It is straightforward to check that
\begin{equation*}
f_m(t)=e^t(t-m)+2\sum_{j=0}^m(m-j)\frac{t^j}{j!}
\end{equation*}
satisfies the recursion \eqref{eq:L:ApprErrorEst3} with equality, so that
\begin{equation*}
\phi_m(t)\leq(t-m)+2\sum_{j=0}^m(m-j)\frac{t^j}{j!}e^{-t}.
\end{equation*}
In view of the estimate
\begin{equation*}
\begin{split}
\sum_{j=0}^\infty\frac{|j-m|m^j\alpha^j}{j!}&\leq e^{\frac{m\alpha}{2}}\sum_{j=0}^\infty\frac{(j-m)^2m^j\alpha^j}{j!}\\
&=e^{m\alpha}\left[m^2(\alpha-1)^2+m(\alpha-1)+m\right]^{1/2}
\end{split}
\end{equation*}
from \cite{miyadera}, we get that
\begin{equation*}
\begin{split}
\phi_m(t)&\leq(t-m)+2\sum_{j=0}^m(j-m)\frac{t^j}{j!}e^{-t}\\
&\leq(t-m)+2e^{-t}\sum_{j=0}^\infty\frac{|j-m|m^j\left(\frac{t}{m}\right)^j}{j!}\\
&\leq(t-m)+2e^{-t}e^{\frac{t}{2}}\sum_{j=0}^\infty\frac{(j-m)^2m^j\left(\frac{t}{m}\right)^j}{j!}\\
&=(t-m)+2\left[(t-m)^2+t\right]^{1/2}.
\end{split}
\end{equation*}
Thus, (ii) follows from the above combined with (i).
\end{proof}

\begin{proposition}\label{P:TrotterApprox}
For $\rho>0$ let $F_\rho:\mathbb{P}\mapsto\mathbb{P}$ be a nonexpansive map, and let $S_\rho(t)$ denote the semigroup generated by $J_{\lambda,\rho}$ corresponding to the nonexpansive map $F:=F_\rho$ for each $\rho>0$ in Theorem~\ref{T:ExponentialFormulaAppr}. If
\begin{equation*}
J_{\lambda,\rho}(X)\to J_\lambda^\mu(X)
\end{equation*}
in $d_\infty$ as $\rho\to 0+$ for a fixed $\mu\in\mathcal{P}^1(\mathbb{P})$ and all $X\in\mathbb{P}$, then
\begin{equation}\label{eq:P:TrotterApprox}
S_\rho(t)X\to S(t)X
\end{equation}
in $d_\infty$ for all $X\in\mathbb{P}$ as $\rho\to 0+$, where $S(t)$ is the semigroup generated by $J_\lambda^\mu$ in Theorem~\ref{T:ExponentialFormula}. Moreover the limit in \eqref{eq:P:TrotterApprox} is uniform on compact time intervals.
\end{proposition}
\begin{proof}
Fix a $T>0$, $X\in\mathbb{P}$ and let $0<t<T$. For all $\lambda>0$ the assumption implies
\begin{equation}\label{eq1:P:TrotterApprox}
\frac{d_\infty(X,J_{\lambda,\rho}(X))}{\lambda}\to \frac{d_\infty(X,J_{\lambda}^\mu(X))}{\lambda}\leq \int_{\mathbb{P}}d_\infty(X,A)d\mu(A)
\end{equation}
as $\rho\to 0+$, where the inequality follows from \eqref{eq:P:ResolventBound}. We also have the following estimates
\begin{equation}\label{eq2:P:TrotterApprox}
\begin{split}
d_\infty(S_\rho(t)X,S(t)X)&\leq d_\infty(S_\rho(t)X,S_\rho(t)J_{\lambda,\rho}(X))+d_\infty(S_\rho(t)J_{\lambda,\rho}(X),S(t)X)\\
&\leq d_\infty(X,J_{\lambda,\rho}(X))+d_\infty(S_\rho(t)J_{\lambda,\rho}(X),S(t)X)
\end{split}
\end{equation}
and
\begin{equation}\label{eq3:P:TrotterApprox}
\begin{split}
d_\infty(S_\rho(t)J_{\lambda,\rho}(X),S(t)X)&\leq d_\infty(S_\rho(t)J_{\lambda,\rho}(X),\left(J_{t/n,\rho}\right)^nJ_{\lambda,\rho}(X))\\
&\quad+d_\infty(\left(J_{t/n,\rho}\right)^nJ_{\lambda,\rho}(X),\left(J_{t/n,\rho}\right)^n(X))\\
&\quad+d_\infty(\left(J_{t/n,\rho}\right)^n(X),\left(J_{t/n}^\mu\right)^n(X))\\
&\quad+d_\infty\left(\left(J_{t/n}^\mu\right)^n(X),S(t)X\right).
\end{split}
\end{equation}
We need to find upper bound on the terms in \eqref{eq3:P:TrotterApprox}. Firstly by \eqref{eq2:T:ExponentialFormulaAppr} we have
\begin{equation*}
\begin{split}
d_\infty(S_{\rho}(t)J_{\lambda,\rho}(X),\left(J_{t/n,\rho}\right)^{n}J_{\lambda,\rho}(X))&\leq \frac{2t}{\sqrt{n}}\frac{d_\infty(J_{\lambda,\rho}(X),F_\rho(J_{\lambda,\rho}(X)))}{\rho}\\
&=\frac{2t}{\sqrt{n}}\frac{d_\infty(X,J_{\lambda,\rho}(X))}{\lambda}
\end{split}
\end{equation*}
where equality follows from \eqref{eq:L:ApproxResolvent1}. Since $J_{\lambda,\rho}$ is a contraction, for all $n\in\mathbb{N}$ we get
\begin{equation*}
d_\infty(\left(J_{t/n,\rho}\right)^nJ_{\lambda,\rho}(X),\left(J_{t/n,\rho}\right)^n(X))\leq d_\infty(J_{\lambda,\rho}(X),X)
\end{equation*}
and by \eqref{eq4:T:ExponentialFormula} we get
\begin{equation*}
d_\infty\left(\left(J_{t/n}^\mu\right)^n(X),S(t)X\right)\leq \frac{2t}{\sqrt{n}}\int_{\mathbb{P}}d_\infty(X,A)d\mu(A).
\end{equation*}
From the above we obtained the following:
\begin{equation}\label{eq4:P:TrotterApprox}
\begin{split}
d_\infty(S_\rho(t)X,S(t)X)&\leq 2d_\infty(J_{\lambda,\rho}(X),X)+\frac{2t}{\sqrt{n}}\frac{d_\infty(X,J_{\lambda,\rho}(X))}{\lambda}\\
&\quad+\frac{2t}{\sqrt{n}}\int_{\mathbb{P}}d_\infty(X,A)d\mu(A)\\
&\quad+d_\infty(\left(J_{t/n,\rho}\right)^n(X),\left(J_{t/n}^\mu\right)^n(X)).
\end{split}
\end{equation}
Now let $\epsilon>0$. Choose $\lambda_0>0$ so that $\lambda_0\int_{\mathbb{P}}d_\infty(X,A)d\mu(A)<\epsilon$. By \eqref{eq1:P:TrotterApprox}, there exists a $\delta>0$ such that for $\rho<\delta$ we have
\begin{equation*}
\frac{d_\infty(X,J_{\lambda_0,\rho}(X))}{\lambda_0}\leq \frac{d_\infty(X,J_{\lambda_0}^\mu(X))}{\lambda_0}+\epsilon.
\end{equation*}
Thus,
\begin{equation*}
d_\infty(X,J_{\lambda_0,\rho}(X))\leq \lambda_0\int_{\mathbb{P}}d_\infty(X,A)d\mu(A)+\epsilon\lambda_0<2\epsilon
\end{equation*}
for $\rho<\delta$. Next, choose an $n_0\in\mathbb{N}$ such that
\begin{equation*}
\frac{2t}{\sqrt{n}}\left[\frac{d_\infty(X,J_{\lambda_0,\rho}(X))}{\lambda}+\int_{\mathbb{P}}d_\infty(X,A)d\mu(A)+\right]<\epsilon
\end{equation*}
for all $n\geq n_0$ and $t<T$.

Finally for $t<T$ we estimate
\begin{equation*}
\begin{split}
d_\infty(\left(J_{t/n_0,\rho}\right)^{n_0}(X),\left(J_{t/n_0}^\mu\right)&^{n_0}(X))\leq d_\infty(\left(J_{t/n_0,\rho}\right)^{n_0}(X),\left(J_{t/n_0,\rho}\right)^{n_0-1}J_{t/n_0}^\mu(X))\\
& +d_\infty(\left(J_{t/n_0,\rho}\right)^{n_0-1}J_{t/n_0}^\mu(X),\left(J_{t/n_0}^\mu\right)^{n_0-1}J_{t/n_0}^\mu(X))
\end{split}
\end{equation*}
which, by induction together with the assumption of the assertion, yields the existence of $\rho_0>0$ such that for $\rho<\rho_0$ we get that
\begin{equation*}
d_\infty(\left(J_{t/n_0,\rho}\right)^{n_0}(X),\left(J_{t/n_0}^\mu\right)^{n_0}(X))<\epsilon.
\end{equation*}
Combining the above with \eqref{eq4:P:TrotterApprox} we obtain \eqref{eq:P:TrotterApprox}.

To show that for a fixed $X\in\mathbb{P}$ the convergence in \eqref{eq:P:TrotterApprox} is uniform for $t<T$, pick $\tau\in(0,T)$. By the triangle inequality and contraction property of $S_\rho(t),S(t)$, \eqref{eq3:T:ExponentialFormulaAppr} and \eqref{eq5:T:ExponentialFormula} we get
\begin{equation*}
\begin{split}
d_\infty(S(t)X,S_\rho(t)X)&\leq d_\infty(S(t)X,S_\rho(t)J_{\lambda,\rho}(X))+d_\infty(S_\rho(t)J_{\lambda,\rho}(X),S_\rho(t)X)\\
&\leq d_\infty(S(t)X,S(\tau)X)+d_\infty(S(\tau)X,S_\rho(\tau)J_{\lambda,\rho}(X))\\
&\quad+d_\infty(S_\rho(\tau)J_{\lambda,\rho}(X),S_\rho(t)J_{\lambda,\rho}(X))+d_\infty(J_{\lambda,\rho}(X),X)\\
&\leq 2|t-\tau|\int_{\mathbb{P}}d_\infty(X,A)d\mu(A)+d_\infty(S(\tau)X,S_\rho(\tau)J_{\lambda,\rho}(X))\\
&\quad+2|t-\tau|\frac{d_\infty(J_{\lambda,\rho}(X),F_\rho(J_{\lambda,\rho}(X)))}{\rho}+d_\infty(J_{\lambda,\rho}(X),X).
\end{split}
\end{equation*}
Now as in the first part, we can fix a $\lambda>0$ such that for sufficiently small $\rho>0$ the quantity $d_\infty(J_{\lambda,\rho}(X),X)$ becomes arbitrarily small. We again have that $\frac{d_\infty(J_{\lambda,\rho}(X),F_\rho(J_{\lambda,\rho}(X)))}{\rho}=\frac{d_\infty(J_{\lambda,\rho}(X),X)}{\lambda}$, and for fixed $\lambda>0$, this term is bounded as $\rho\to 0+$. We have also seen before that $d_\infty(S(\tau)X,S_\rho(\tau)J_{\lambda,\rho}(X))$ is small for small $\rho>0$. Now we can use the compactness of $[0,T]$ to conclude that the convergence is uniform on $[0,T]$ in \eqref{eq:P:TrotterApprox}.
\end{proof}

\begin{theorem}\label{T:TrotterFormula}
For each $\rho>0$ let $F_\rho:\mathbb{P}\mapsto\mathbb{P}$ be a nonexpansive map and let $J_{\lambda,\rho}$ be the resolvent generated by $F_\rho$ in \eqref{eq:L:ApproxResolvent1} for each $\rho>0$. If
\begin{equation*}
J_{\lambda,\rho}(X)\to J_\lambda^\mu(X)
\end{equation*}
in $d_\infty$ as $\rho\to 0+$ for a fixed $\mu\in\mathcal{P}^1(\mathbb{P})$ and all $X\in\mathbb{P}$, then
\begin{equation}\label{eq:T:TrotterFormula}
(F_{\frac{t}{n}})^n(X)\to S(t)X
\end{equation}
in $d_\infty$ for all $X\in\mathbb{P}$ as $n\to \infty$, where $S(t)$ is the semigroup generated by $J_\lambda^\mu$ in Theorem~\ref{T:ExponentialFormula}. Moreover the limit in \eqref{eq:T:TrotterFormula} is uniform on compact time intervals.
\end{theorem}
\begin{proof}
Fix $T>0$, let $X\in\mathbb{P}$ and let $0<t\leq T$. Let $S_\rho(t)$ denote the semigroup generated by $J_{\lambda,\rho}$. We have
\begin{equation}\label{eq1:T:TrotterFormula}
d_\infty(S(t)X,(F_{\frac{t}{n}})^n(X))\leq d_\infty(S(t)X,S_{\frac{t}{n}}(t)X)+d_\infty(S_{\frac{t}{n}}(t)X,(F_{\frac{t}{n}})^n(X)).
\end{equation}
For $\rho>0$, $n\in\mathbb{N}$ and $\lambda>0$ we have
\begin{equation}\label{eq2:T:TrotterFormula}
\begin{split}
d_\infty(S_{\rho}(n\rho)X,(F_{\rho})^n(X))&\leq d_\infty(S_{\rho}(n\rho)X,S_{\rho}(n\rho)J_{\lambda,\rho}(X))\\
&\quad+d_\infty(S_{\rho}(n\rho)J_{\lambda,\rho}(X),(F_{\rho})^n(J_{\lambda,\rho}(X)))\\
&\quad+d_\infty((F_{\rho})^n(J_{\lambda,\rho}(X)),(F_{\rho})^n(X))\\
&\leq 2d_\infty(J_{\lambda,\rho}(X),X)\\
&\quad+d_\infty(S_{\rho}(n\rho)J_{\lambda,\rho}(X),(F_{\rho})^n(J_{\lambda,\rho}(X))).
\end{split}
\end{equation}
For $\rho=\frac{t}{n}$ and $\lambda>0$ we have by Lemma~\ref{L:ApprErrorEst} that
\begin{equation}\label{eq3:T:TrotterFormula}
\begin{split}
d_\infty(S_{\rho}(n\rho)J_{\lambda,\rho}(X)&,(F_{\rho})^n(J_{\lambda,\rho}(X)))\\
&\leq 2\sqrt{\left(n-\frac{n\rho}{\rho}\right)^2+\frac{n\rho}{\rho}}d_\infty(J_{\lambda,\rho}(X),F_\rho(J_{\lambda,\rho}(X)))\\
&=2\sqrt{n}\rho\frac{d_\infty(J_{\lambda,\rho}(X),X)}{\lambda}\leq 2\frac{T}{\sqrt{n}}\frac{d_\infty(J_{\lambda,\rho}(X),X)}{\lambda}
\end{split}
\end{equation}
where the equality follows from \eqref{eq:L:ApproxResolvent1}.

We have already seen in the proof of Proposition~\ref{P:TrotterApprox} after \eqref{eq4:P:TrotterApprox} how the convergence of the resolvents $J_{\lambda,\rho}(X)\to J_\lambda^\mu(X)$ imply estimates on $d_\infty(J_\lambda^\mu(X),X)$, $d_\infty(J_{\lambda,\rho}(X),X)$ and $\frac{d_\infty(J_{\lambda,\rho}(X),X)}{\lambda}$. Thus the estimates \eqref{eq1:T:TrotterFormula}, \eqref{eq2:T:TrotterFormula} and \eqref{eq3:T:TrotterFormula} along with Proposition~\ref{P:TrotterApprox} implies uniform convergence in \eqref{eq:T:TrotterFormula} on compact time intervals.
\end{proof}

\section{Convergence of resolvents}

In this section we prove the convergence of the resolvents $J_{\lambda,\rho}(X)\to J_\lambda^\mu(X)$ in $d_\infty$ for finitely supported measures $\mu=\sum_{i=1}^n\frac{1}{n}\delta_{A_i}\in\mathcal{P}^1(\mathbb{P})$ in Theorem~\ref{T:ResolventConv}. This result will play a key role later in proving a continuous time law of large numbers type result for $\Lambda$. The analysis of this section also proves the norm convergence of power means to the Karcher mean solving this conjecture mentioned in \cite{lawsonlim1}.

Let $\mathbb{P}^*$ denote the dual cone of $\mathbb{P}$, i.e. the cone of all non-negative norm continuous linear functionals on $\mathbb{P}$.

\begin{lemma}\label{L:NormingLinearFunct}
Let $A,B\in\mathbb{P}$. Then there exists an $\omega\in\mathbb{P}^*$ with $\omega(I)=1$ such that either
\begin{equation*}
e^{d_\infty(B\#_tA,B)}\omega(B)=\omega(B\#_tA)
\end{equation*}
or
\begin{equation*}
\omega(B\#_tA)=e^{d_\infty(B\#_tA,B)}\omega(B)
\end{equation*}
holds for all $t\in[0,1]$.
\end{lemma}
\begin{proof}
By definition we have that
$$d_\infty(A,B)=\log\max\{\inf\{\alpha>0:A\leq\alpha B\},\inf\{\alpha>0:B\leq\alpha A\}\}.$$
Assume first that $B=I$. Then
\begin{equation*}
e^{d_\infty(I\#_tA,I)}=\max\{\|A^t\|,\|A^{-t}\|\},
\end{equation*}
since we have
\begin{equation*}
\inf\{\alpha>0:A\leq\alpha I\}=\|A\|.
\end{equation*}
We also have that there exists a net $v_\alpha\in\mathcal{H}$ with $|v_\alpha|=1$ and $\lim_{\alpha\to\infty}|Av_\alpha|=\|A\|$. That is
\begin{equation*}
\|A\|^2=\lim_{\alpha\to\infty}v_\alpha^*A^*Av_\alpha=\lim_{\alpha\to\infty}v_\alpha^*A^2v_\alpha=\lim_{\alpha\to\infty}\omega_\alpha(A^2),
\end{equation*}
where $v_\alpha^*(\cdot)v_\alpha=:\omega_\alpha\in\mathbb{P}^*$ and $\omega_\alpha(I)=\|\omega_\alpha\|_{*}=1$. Since the convex set $\{\nu\in\mathbb{P}^*:\nu(I)=\|\nu\|_{*}=1\}$ is weak-$*$ compact, there exists a subnet of $\omega_\alpha$ again denoted by $\omega_\alpha$ that has a limit point $\omega\in\{\nu\in\mathbb{P}^*:\nu(I)=\|\nu\|_{*}=1\}$. Then the state $\omega$ satisfies $\|A\|^2=\omega(A^2)$ which is equivalent to $\|A\|=\omega(A)$, more generally by the monotonicity of the power function it follows that
\begin{equation*}
\|A^t\|=\omega(A^t).
\end{equation*}
If $e^{d_\infty(A,I)}=\|A\|$, this yields that
\begin{equation*}
\omega(I\#_tA)=\omega(A^t)=e^{d_\infty(I\#_tA,I)}\omega(I).
\end{equation*}
In the other case when $e^{d_\infty(A,I)}=\|A^{-1}\|$ by the same argument as above, we can find an $\omega\in\{\nu\in\mathbb{P}^*:\nu(I)=\|\nu\|_{*}=1\}$ such that
\begin{equation*}
e^{-d_\infty(I\#_tA,I)}\omega(I)=\omega(I\#_tA).
\end{equation*}
Now the case of arbitrary $B\in\mathbb{P}$ follows by considering first
\begin{equation*}
\omega\left(\left(B^{-1/2}AB^{-1/2}\right)^t\right)=e^{d_\infty\left(\left(B^{-1/2}AB^{-1/2}\right)^t,I\right)}\omega(I)
\end{equation*}
which is equivalent to
\begin{equation*}
\hat{\omega}(B\#_tA)=e^{d_\infty(I\#_t(B^{-1/2}AB^{-1/2}),I)}\hat{\omega}(B)=e^{d_\infty(B\#_tA,B)}\hat{\omega}(B)
\end{equation*}
where $\hat{\omega}(X):=\frac{1}{\omega(B^{-1})}\omega(B^{-1/2}XB^{-1/2})$. The other equality in the assertion follows similarly from the case $B=I$.
\end{proof}

We will use the notation
$$\phi_\mu(X):=\int_{\mathbb{P}}\log_XAd\mu(A)$$
for $\mu\in\mathcal{P}^1(\mathbb{P})$ and $X\in\mathbb{P}$. In the remaining parts of this section we assume that the map $\phi_\mu:\mathbb{P}\mapsto\mathbb{S}$ is Fr\'{e}chet-differentiable, for example this is the case if $\mu$ is finitely supported.

\begin{proposition}\label{P:FrechetDBounded}
Let $\mu\in\mathcal{P}^1(\mathbb{P})$ and $X\in\mathbb{P}$. Let $D\phi_\mu(X)[V]$ denote the Fr\'{e}chet-derivative of $\phi_\mu$ in the direction $V\in\mathbb{S}$. Then the linear map $D\phi_\mu(\Lambda(\mu)):\mathbb{S}\mapsto\mathbb{S}$ is injective, in particular
\begin{equation}\label{eq:P:FrechetDBounded}
\frac{1}{\|\Lambda(\mu)^{-1}\|\|\Lambda(\mu)\|}\leq \|D\phi_\mu(\Lambda(\mu))\|
\end{equation}
where $\|D\phi_\mu(\Lambda(\mu))\|:=\sup_{V\in\mathbb{S},\|V\|=1}\|D\phi_\mu(\Lambda(\mu))[V]\|$.
\end{proposition}
\begin{proof}
Let $X,Y\in\mathbb{P}$ and according to \eqref{eq1:T:ExponentialFormula} let $\gamma(t):=S(t)X$, $\eta(t):=S(t)Y$. Then by \eqref{eq2:T:ExponentialFormula} we have
\begin{equation}\label{eq1:P:FrechetDBounded}
d_\infty\left(\gamma(t),\eta(t)\right)\leq e^{-t}d_\infty(X,Y).
\end{equation}
Then by Lemma~\ref{L:NormingLinearFunct} there exists an $\omega\in\mathbb{P}^*$ with $\omega(I)=1$ such that either
\begin{equation}\label{eq2:P:FrechetDBounded}
\omega(X)=e^{d_\infty(X,Y)}\omega(Y)
\end{equation}
or
\begin{equation}\label{eq2.1:P:FrechetDBounded}
\omega(Y)=e^{d_\infty(X,Y)}\omega(X).
\end{equation}
Assume that \eqref{eq2:P:FrechetDBounded} holds. In general we have that
\begin{equation*}
\gamma(t)\leq e^{d_\infty\left(\gamma(t),\eta(t)\right)}\eta(t)
\end{equation*}
which combined with \eqref{eq1:P:FrechetDBounded} yields
\begin{equation*}
\gamma(t)\leq e^{e^{-t}d_\infty(X,Y)}\eta(t).
\end{equation*}
Thus, since $\omega$ is positive we have
\begin{equation}\label{eq3:P:FrechetDBounded}
\omega(\gamma(t))\leq e^{e^{-t}d_\infty(X,Y)}\omega(\eta(t))
\end{equation}
where for $t=0$ we have equality by \eqref{eq2:P:FrechetDBounded}. Hence we may take the derivative of \eqref{eq3:P:FrechetDBounded} at $t=0$ to get
\begin{equation}\label{eq4:P:FrechetDBounded}
\omega(\dot{\gamma}(0))\leq e^{d_\infty(X,Y)}\omega(\dot{\eta}(0))-d_\infty(X,Y)e^{d_\infty(X,Y)}\omega(\eta(0)).
\end{equation}
If \eqref{eq2.1:P:FrechetDBounded} holds then we start from
\begin{equation*}
\eta(t)\leq e^{d_\infty\left(\gamma(t),\eta(t)\right)}\gamma(t)
\end{equation*}
to obtain
\begin{equation}\label{eq4.1:P:FrechetDBounded}
\omega(\dot{\eta}(0))\leq e^{d_\infty(X,Y)}\omega(\dot{\gamma}(0))-d_\infty(X,Y)e^{d_\infty(X,Y)}\omega(\gamma(0)).
\end{equation}
by a similar argument.

Now assume that $\eta(0)=\Lambda(\mu)$ and $\gamma(0)=\eta(0)\#_sZ$ for $Z\in\mathbb{P}$ and $s\in[0,1]$. Then by Theorem~\ref{T:StrongSolution} we have that
\begin{equation}\label{eq5:P:FrechetDBounded}
\begin{split}
\dot{\eta}(0)&=\phi_\mu(\Lambda(\mu))=0,\\
\dot{\gamma}(0)&=\phi_\mu(\gamma(0)).
\end{split}
\end{equation}
First assume that \eqref{eq4:P:FrechetDBounded} holds so that by \eqref{eq5:P:FrechetDBounded} and Lemma~\ref{L:NormingLinearFunct} we have
\begin{equation*}
\begin{split}
\omega(\phi_\mu(\gamma(0)))&\leq -d_\infty(\eta(0),\eta(0)\#_sZ)e^{d_\infty(\eta(0),\eta(0)\#_sZ)}\omega(\eta(0))\\
&=-sd_\infty(\Lambda(\mu),Z)e^{sd_\infty(\Lambda(\mu),Z)}\omega(\Lambda(\mu))
\end{split}
\end{equation*}
for all $s\in[0,1]$. Since $\phi_\mu(\gamma(0))=\phi_\mu(\Lambda(\mu)\#_sZ)$, the above yields
\begin{equation}\label{eq6:P:FrechetDBounded}
\omega(\phi_\mu(\Lambda(\mu)\#_sZ))\leq -sd_\infty(\Lambda(\mu),Z)e^{sd_\infty(\Lambda(\mu),Z)}\omega(\Lambda(\mu)).
\end{equation}
For $s=0$ we have equality in \eqref{eq6:P:FrechetDBounded} since the two sides both equal to $0$, thus we can differentiate \eqref{eq6:P:FrechetDBounded} at $s=0$ to get
\begin{equation*}
\omega(D\phi_\mu(\Lambda(\mu))[\log_{\Lambda(\mu)}(Z)])\leq -d_\infty(\Lambda(\mu),Z)\omega(\Lambda(\mu))
\end{equation*}
and it follows that
\begin{equation}\label{eq7:P:FrechetDBounded}
\omega(\Lambda(\mu))\leq -\omega\left(D\phi_\mu(\Lambda(\mu))\left[\frac{1}{\|\log(\Lambda(\mu)^{-1/2}Z\Lambda(\mu)^{-1/2})\|}\log_{\Lambda(\mu)}(Z)\right]\right).
\end{equation}
In the other case when \eqref{eq4.1:P:FrechetDBounded} holds, by a similar argument we obtain
\begin{equation*}
d_\infty(\Lambda(\mu),Z)\omega(\Lambda(\mu))\leq \omega(D\phi_\mu(\Lambda(\mu))[\log_{\Lambda(\mu)}(Z)])
\end{equation*}
and thus
\begin{equation}\label{eq7.1:P:FrechetDBounded}
\omega(\Lambda(\mu))\leq \omega\left(D\phi_\mu(\Lambda(\mu))\left[\frac{1}{\|\log(\Lambda(\mu)^{-1/2}Z\Lambda(\mu)^{-1/2})\|}\log_{\Lambda(\mu)}(Z)\right]\right).
\end{equation}

As $Z$ ranges over $\mathbb{P}$, the expression $\log_{\Lambda(\mu)}(Z)$ attains all possible values in $\mathbb{S}$ for which either we have \eqref{eq7:P:FrechetDBounded} or \eqref{eq7.1:P:FrechetDBounded} moreover $\omega(\Lambda(\mu))>0$ since $\omega\in\mathbb{P}^*$ and $\omega(I)=1$ and there exists $\epsilon>0$ such that $\epsilon I\leq \Lambda(\mu)$, thus $D\phi_\mu(\Lambda(\mu)):\mathbb{S}\mapsto\mathbb{S}$ is an injective bounded linear map. The inequality \eqref{eq:P:FrechetDBounded} follows \eqref{eq7:P:FrechetDBounded} or \eqref{eq7.1:P:FrechetDBounded}.
\end{proof}

We need a few facts from the theory of bounded linear operators. We denote the predual of the von Neumann algebra $\mathcal{B}(\mathcal{H})$ by $\mathcal{B}(\mathcal{H})_{*}$ which is the ideal of trace class operators on $\mathcal{H}$. If we restrict to the self-adjoint part $\mathbb{S}$, then $\mathbb{S}$ is a real Banach space with predual $\mathbb{S}_*:=\{X\in\mathcal{B}(\mathcal{H})_{*},X^*=X\}$ and dual space $\mathbb{S}^*:=\{X\in\mathcal{B}(\mathcal{H})^{*},X^*=X\}$ where $\mathcal{B}(\mathcal{H})^{*}$ denotes the predual of the \emph{universal enveloping von Neumann algebra} $\mathcal{B}(\mathcal{H})^{**}$. Since $\mathbb{S}^*$ is the self-adjoint part of the ideal of trace-class operators $\mathcal{B}(\mathcal{H})^{*}$ in the universal enveloping von Neumann algebra $\mathcal{B}(\mathcal{H})^{**}$ and thus is the unique predual of the von Neumann algebra $\mathcal{B}(\mathcal{H})^{**}$, we have that any given $X\in\mathbb{S}^*$ can be uniquely decomposed as $X=X^{+}-X^{-}$ where $X^{+},X^{-}\geq 0$ and such that the support projections of $X^{+}$ and $X^{-}$ are orthogonal by Theorem III.4.2. in \cite{takesaki}. The locally convex topology $\sigma(\mathbb{S}_*,\mathbb{S})$ is called the $\sigma$-weak or ultraweak operator topology on $\mathbb{S}$.

\begin{lemma}\label{L:FrechetDenseRange}
Let $\mu\in\mathcal{P}^1(\mathbb{P})$ and $X\in\mathbb{P}$. Then the linear map $D\phi_\mu(\Lambda(\mu)):\mathbb{S}\mapsto\mathbb{S}$ has dense range in the ultraweak operator topology.
\end{lemma}
\begin{proof}
Firstly, notice that for $X,A\in\mathbb{P}$ and invertible $C\in\mathcal{B}(\mathcal{H})$ we have
\begin{equation*}
\begin{split}
C\log_XAC^*&=CX\log(X^{-1}A)C^*\\
&=CXC^*C^{-*}\log(X^{-1}A)C^*\\&
=CXC^*\log(C^{-*}X^{-1}C^{-1}CAC^*)\\
&=\log_{CXC^*}(CAC^*),
\end{split}
\end{equation*}
hence $CD\log_XA[V]C^*=D\log_{CXC^*}(CAC^*)[CVC^*]$, where the differentiation is with respect to the variable $X$ and $A$ is fixed. Thus for arbitrary $V\in\mathbb{S}$ it also follows that
\begin{equation*}
D\phi_\mu(\Lambda(\mu))[V]=\Lambda(\mu)^{1/2}D\phi_{\hat{\mu}}(I)[\Lambda(\mu)^{-1/2}V\Lambda(\mu)^{-1/2}]\Lambda(\mu)^{1/2}
\end{equation*}
where $d\hat{\mu}(A):=d\mu(\Lambda(\mu)^{1/2}A\Lambda(\mu)^{1/2})$ and $\Lambda(\hat{\mu})=I$ as well. Thus it is enough to prove that the range of $D\phi_\mu(\Lambda(\mu))$ is dense when $\Lambda(\mu)=I$.

So without loss of generality assume that $\Lambda(\mu)=I$. It is well known that on a locally convex space $X$, a linear operator $T:X\mapsto X$ has dense range if and only if there are no nonzero linear functionals in the dual space $X^*$ which vanish on the range of the map $T$. So assume on the contrary that there exists a nonzero ultraweakly continuous linear functional $\tau\in\mathbb{S}_*$ such that $\tau(D\phi_\mu(\Lambda(\mu))[V])=0$ for all $V\in\mathbb{S}$. In what follows we will reverse the construction given in the proof of Lemma~\ref{L:NormingLinearFunct}. Consider the unique decomposition $\tau=\tau_+-\tau_-$ where $\tau_+,\tau_-\geq 0$ and the support projections $s(\tau_+),s(\tau_-)$ of $\tau_+,\tau_-$ are orthogonal. Let $X:=\exp(s(\tau_+)-s(\tau_-))$. Then by the orthogonality of $s(\tau_+),s(\tau_-)$ we have
\begin{equation*}
X=\exp(s(\tau_+))\oplus \exp(-s(\tau_-))\oplus I_{E/(\mathrm{Rg}(s(\tau_+))\cup \mathrm{Rg}(s(\tau_-)))}
\end{equation*}
if we restrict the domain of $s(\tau_+),s(\tau_-)$ to their range respectively. Consider the states $\hat{\tau}_+,\hat{\tau}_-\in\mathbb{S}^*$ defined as
\begin{equation*}
\begin{split}
\hat{\tau}_+(\cdot)&:=\frac{1}{\tau_+(I)}\tau_+(\cdot),\\
\hat{\tau}_-(\cdot)&:=\frac{1}{\tau_-(I)}\tau_-(\cdot).
\end{split}
\end{equation*}
By construction it follows that they are both norming linear functionals for $X$ in the following sense:
\begin{equation*}
\begin{split}
\hat{\tau}_+(X)&=e,\\
\hat{\tau}_-(X^{-1})&=e.
\end{split}
\end{equation*}
So we can follow the argumentation from \eqref{eq2:P:FrechetDBounded} with $Y=I$ and the state $\omega:=\hat{\tau}_+$ to arrive at \eqref{eq7:P:FrechetDBounded}, that is
\begin{equation}\label{eq1:L:FrechetDenseRange}
\hat{\tau}_+(I)\leq -\hat{\tau}_+\left(D\phi_\mu(I)\left[\frac{1}{\|\log(X)\|}\log(X)\right]\right).
\end{equation}
Similarly, in the other case we choose $\omega:=\hat{\tau}_-$ in \eqref{eq2.1:P:FrechetDBounded} to obtain \eqref{eq7.1:P:FrechetDBounded}, which is
\begin{equation}\label{eq2:L:FrechetDenseRange}
\hat{\tau}_-(I)\leq \hat{\tau}_-\left(D\phi_\mu(I)\left[\frac{1}{\|\log(X)\|}\log(X)\right]\right).
\end{equation}
We have that $\tau=\tau_+-\tau_-$, hence \eqref{eq1:L:FrechetDenseRange} and \eqref{eq2:L:FrechetDenseRange} yields
\begin{equation*}
0<\tau_+(I)\hat{\tau}_+(I)+\tau_-(I)\hat{\tau}_-(I)\leq-\tau\left(D\phi_\mu(I)\left[\frac{1}{\|\log(X)\|}\log(X)\right]\right)
\end{equation*}
contradicting the initial assumption $\tau(D\phi_\mu(\Lambda(\mu))[V])=0$ for all $V\in\mathbb{S}$.
\end{proof}

\begin{lemma}\label{L:FrechetCommutesRepr}
Let $\mu\in\mathcal{P}^1(\mathbb{P})$ and $X\in\mathbb{P}$. Let $\pi:\mathcal{B}(\mathcal{H})\mapsto\mathcal{A}$ be a unital $*$-representation into a unital $C^*$-algebra $\mathcal{A}$. Then the linear map $D\phi_\mu(\Lambda(\mu)):\mathbb{S}\mapsto\mathbb{S}$ commutes with $\pi$, i.e.
\begin{equation*}
\pi\left(D\phi_\mu(\Lambda(\mu))[V]\right)=D\phi_{\hat{\mu}}(\Lambda(\hat{\mu}))[\pi(V)]
\end{equation*}
where $d\hat{\mu}(A):=d\mu(\pi^{-1}(A))$. Moreover the linear map $D\phi_\mu(\Lambda(\mu)):\mathbb{S}\mapsto\mathbb{S}$ is ultraweakly continuous.
\end{lemma}
\begin{proof}
Firstly, since $\pi$ is a $*$-representation it is automatically norm continuous, hence $d\hat{\mu}(A)$ is well defined by the inverse image of the continuous map $\pi$. Secondly the continuous function $\log$ can be defined as
\begin{equation*}
\log(X)=\int_0^\infty\frac{\lambda}{\lambda^2+1}I-(\lambda I+X)^{-1}d\lambda,
\end{equation*}
thus
\begin{equation}\label{eq1:L:FrechetCommutesRepr}
D\log(X)[V]=\int_0^\infty(\lambda I+X)^{-1}V(\lambda I+X)^{-1}d\lambda,
\end{equation}
where both integrals converge in the norm topology. Then it is easy to see that
\begin{equation*}
\pi\left(D\log(X)[V]\right)=D\log(\pi(X))[\pi(V)]
\end{equation*}
and similarly $\Lambda(\cdot)$ and $D\phi_\mu(\Lambda(\mu))[\cdot]$ commutes with $\pi$ as well.

The ultraweak continuity of $D\phi_\mu(\Lambda(\mu)):\mathbb{S}\mapsto\mathbb{S}$ can be deduced from the formula
\begin{equation*}
\begin{split}
D\phi_\mu(\Lambda(\mu))[V]=-\Lambda(\mu)\int_{\mathbb{P}}\int_0^\infty&\left(\lambda I+\Lambda(\mu)^{-1}A\right)^{-1}\Lambda(\mu)^{-1}V\Lambda(\mu)^{-1}A\\
&\times\left(\lambda I+\Lambda(\mu)^{-1}A\right)^{-1}d\lambda d\mu(A)
\end{split}
\end{equation*}
which is derived using \eqref{eq1:L:FrechetCommutesRepr}.
\end{proof}

\begin{proposition}\label{P:FrechetDIso}
Let $\mu\in\mathcal{P}^1(\mathbb{P}(\mathcal{H}))$ and $X\in\mathbb{P}(\mathcal{H})$. Then the linear map $D\phi_\mu(\Lambda(\mu)):\mathbb{S}(\mathcal{H})\mapsto\mathbb{S}(\mathcal{H})$ is a Banach space isomorphism.
\end{proposition}
\begin{proof}
By Proposition~\ref{P:FrechetDBounded} we know that $D\phi_\mu(\Lambda(\mu))$ is injective and bounded below, hence its range is norm closed. As before, we assume without loss of generality that $\Lambda(\mu)=I$. Thus it remains to show that $D\phi_\mu(I):\mathbb{S}(\mathcal{H})\mapsto\mathbb{S}(\mathcal{H})$ has norm dense range.

So, on the contrary assume that the range of $D\phi_\mu(I):\mathbb{S}(\mathcal{H})\mapsto\mathbb{S}(\mathcal{H})$ is not norm dense. Then there exists a nonzero norm continuous linear functional $\omega\in\mathbb{S}(\mathcal{H})^*$ such that
\begin{equation}\label{eq1:P:FrechetDIso}
\omega\left(D\phi_\mu(I)[V]\right)=0
\end{equation}
for all $V\in\mathbb{S}(\mathcal{H})$. Let $\pi_u:\mathcal{B}(\mathcal{H})\mapsto\mathcal{B}(\mathcal{H}_{\pi_u})$ denote the \emph{universal representation} on the GNS direct sum Hilbert space $\mathcal{H}_{\pi_u}$. Notice that $\pi_u$ is unital. Notice that $\omega\in\mathcal{B}(\mathcal{H}_{\pi_u})_{*}$, hence by Lemma~\ref{L:FrechetCommutesRepr} we have
\begin{equation}\label{eq2:P:FrechetDIso}
\begin{split}
0=\omega\left(D\phi_\mu(I)[V]\right)&=\tr\left\{\omega \pi_u\left(D\phi_\mu(I)[V]\right)\right\}\\
&=\tr\left\{\omega D\phi_{\hat{\mu}}(\pi_u(I))[\pi_u(V)]\right\}\\
&=\tr\left\{\omega D\phi_{\hat{\mu}}(I)[\pi_u(V)]\right\}
\end{split}
\end{equation}
where $d\hat{\mu}(A):=d\mu(\pi_u^{-1}(A))$. We also know that the range of $\pi_u$ is ultraweakly dense in the universal enveloping von Neumann algebra $\mathcal{B}(\mathcal{H})^{**}$, hence by the ultraweak continuity in Lemma~\ref{L:FrechetCommutesRepr} the map $D\phi_{\hat{\mu}}(I)[\pi_u(\cdot)]:\mathbb{S}(\mathcal{H})\mapsto\mathbb{S}(\mathcal{H}_{\pi_u})$ ultraweak continuously extends to the linear map $D\phi_{\hat{\mu}}(I):\mathbb{S}(\mathcal{H}_{\pi_u})\mapsto\mathbb{S}(\mathcal{H}_{\pi_u})$. Then by the ultraweak continuity of $\omega$ on $\mathcal{B}(\mathcal{H}_{\pi_u})$ and \eqref{eq2:P:FrechetDIso} we get that
\begin{equation*}
\omega\left(D\phi_{\hat{\mu}}(I)[Z]\right)=\tr\left\{\omega D\phi_{\hat{\mu}}(I)[Z]\right\}=0
\end{equation*}
for all $Z\in\mathbb{S}(\mathcal{H}_{\pi_u})$. This means that $\omega$ on $\mathbb{S}(\mathcal{H}_{\pi_u})$ is a nonzero ultraweakly continuous linear functional vanishing on the range of $D\phi_{\hat{\mu}}(I):\mathbb{S}(\mathcal{H}_{\pi_u})\mapsto\mathbb{S}(\mathcal{H}_{\pi_u})$ contradicting the ultraweak density of the range of $D\phi_{\hat{\mu}}(I)$ proved in Lemma~\ref{L:FrechetDenseRange}.
\end{proof}

As a warm-up to more involved computations to follow
 we prove the norm convergence conjecture of the power means to the Karcher mean
  mentioned first in \cite{lawsonlim1}. The conjecture states that $\lim_{t\to 0+}P_t(\mu)=\Lambda(\mu)$
   in the norm topology, where $\mu\in\mathcal{P}^1(\mathbb{P})$ is finitely supported. More generally one can assume that the integral in
   \eqref{eq:P:PowerMeans} is bounded for all $t\in[0,1]$. That is the case if all moments
   $\int_{\mathbb{P}}d_\infty^p(X,A)d\mu(A)<+\infty$ for all $p\geq 1$ and $X\in\mathbb{P}$.


\begin{lemma}\label{L:PowerCont}
Let $\mu\in\mathcal{P}^1(\mathbb{P})$ with $\int_{\mathbb{P}}d_\infty^p(X,A)d\mu(A)<+\infty$ for all $p\geq 1$ and $X\in\mathbb{P}$. Consider the function $F:[-1,1]\times\mathbb{P}\mapsto\mathbb{S}$ defined as
\begin{equation}\label{eq:L:PowerCont}
F(t,X):=
\begin{cases}
\int_{\mathbb{P}}\frac{1}{t}[X\#_tA-X]d\mu(A),&\text{if }t\neq 0,\\
\int_{\mathbb{P}}\log_XAd\mu(A),&\text{if }t=0.
\end{cases}
\end{equation}
Then $F$ and its Fr\'echet derivative $DF[\cdot]$ with respect to the variable $X$ is a norm continuous function if we equip the product space $[-1,1]\times\mathbb{P}$ with the max norm generated by the individual Banach space norms on each factor.
\end{lemma}
\begin{proof}
The function $F$ is a smooth function everywhere except when $t=0$, so we have to consider only this case.

The norm continuity of $F$ follows easily from the fact that
\begin{equation*}
\lim_{t\to 0}\frac{1}{t}[X\#_tA-X]=\log_XA.
\end{equation*}
We also have for $t\neq 0$ and $V\in\mathbb{S}$ that
\begin{equation*}
\begin{split}
D\left(\frac{1}{t}[X\#_tA-X]\right)&[V]=\frac{1}{t}D\left(X\left(X^{-1}A\right)^t-X\right)\\
&=\frac{1}{t}\left\{V\left((X^{-1}A)^t-I\right)\right.\\
&\quad\left.+XD\exp\left(t\log(X^{-1}A)\right)\left[tD\log(X^{-1}A)[-X^{-1}VX^{-1}A]\right]\right\}\\
&=V\frac{1}{t}\left((X^{-1}A)^t-I\right)\\
&\quad+XD\exp\left(t\log(X^{-1}A)\right)\left[D\log(X^{-1}A)[-X^{-1}VX^{-1}A]\right].
\end{split}
\end{equation*}
Taking the limit $t\to 0$ in the above, we obtain
\begin{equation*}
\lim_{t\to 0}D\left(\frac{1}{t}[X\#_tA-X]\right)[V]=V\log(X^{-1}A)+XD\log(X^{-1}A)[-X^{-1}VX^{-1}A],
\end{equation*}
i.e. we derived that
\begin{equation*}
D\log_XA[V]=\lim_{t\to 0}D\left(\frac{1}{t}[X\#_tA-X]\right)[V]
\end{equation*}
where differentiation is with respect to the variable $X$. Moreover it is easy to see, that the limit above is uniform for $\|V\|\leq 1$.
\end{proof}

\begin{theorem}[Continuity of $P_t$]\label{T:PowerNormCont}
Let $\mu\in\mathcal{P}^1(\mathbb{P})$ with $\int_{\mathbb{P}}d_\infty^p(X,A)d\mu(A)<+\infty$ for all $p\geq 1$ and $X\in\mathbb{P}$. Then the family $P_t(\mu)$ is norm continuous in $t\in[-1,1]$, in particular
\begin{equation}\label{eq:T:PowerNormCont}
\Lambda(\mu)=\lim_{t\to 0}P_t(\mu)
\end{equation}
in norm.
\end{theorem}
\begin{proof}
We will use the Banach space version of the implicit function theorem, see for instance Theorem 4.9.3 in \cite{loomis}.

Consider the one-parameter family of functions $F:[-1,1]\times\mathbb{P}\mapsto\mathbb{S}$ defined in \eqref{eq:L:PowerCont}. Then by Lemma~\ref{L:PowerCont} the function $F$ and its Fr\'echet derivative $DF$ with respect to the second variable is continuous in the norm topology. Moreover $F(0,\Lambda(\mu))=0$ and $DF(0,\Lambda(\mu))[0,\cdot]$ is a Banach space isomorphism by Proposition~\ref{P:FrechetDIso}. Therefore by the implicit function theorem (Theorem 4.9.3 \cite{loomis}) there exists an open interval $(a,-a)$ of $0\in[-1,1]$ and a norm continuous function $\hat{P}(t)$ such that the operator equation
\begin{equation}\label{eq1:T:PowerNormCont}
F(t,\hat{P}(t))=0
\end{equation}
is satisfied on $(-a,a)$, moreover it is uniquely satisfied there by the function $\hat{P}(t)$. Thus it follows by Proposition~\ref{P:PowerMeans} that $\hat{P}(t)=P_t(\mu)$ and also $\hat{P}(0)=\Lambda(\mu)$. Since $\hat{P}(t)$ varies continuously in $(-a,a)$, therefore $P_t$ does as well.
\end{proof}

The following convergence result is essential for proving the Trotter-type convergence formula for approximation semigroups.

\begin{theorem}\label{T:ResolventConv}
Let $\mu=\sum_{i=1}^n\frac{1}{n}\delta_{A_i}\in\mathcal{P}^1(\mathbb{P})$. For $\rho>0$ let $F_\rho:=J_{\rho/n}^{\delta_{A_n}}\circ\cdots\circ J_{\rho/n}^{\delta_{A_1}}$ where $J_{\rho}^{\delta_{A}}(X):=X\#_{\frac{\rho}{\rho+1}}A$ in the spirit of \eqref{eq:D:resolvent}. In particular $F_\rho:\mathbb{P}\mapsto\mathbb{P}$ is a contraction with respect to $d_\infty$. For $\lambda>0$ let $J_{\lambda,\rho}$ denote the approximating resolvent corresponding to $F_{\rho}$ defined in \eqref{eq:L:ApproxResolvent1}. Then
\begin{equation}\label{eq:T:ResolventConv}
J_{\lambda}^{\mu}(X)=\lim_{\rho\to 0+}J_{\lambda,\rho}(X)
\end{equation}
in norm, where $J_{\lambda}^{\mu}$ is defined by \eqref{eq:D:resolvent}.
\end{theorem}
\begin{proof}
The proof in principle is similar to the proof of Theorem~\ref{T:PowerNormCont} in the sense that we will use the implicit function theorem in the same way. Fix an $X\in\mathbb{P}$ and let $\nu:=\mu+\frac{1}{\lambda}\delta_X$. Consider the function $F:\mathbb{R}\times\mathbb{P}\mapsto\mathbb{S}$ defined as
\begin{equation}\label{eq1:T:ResolventConv}
F(\rho,Y):=
\begin{cases}
\frac{1}{\rho}[J_{\rho/\lambda}^{\delta_{X}}\circ F_\rho(Y)-Y],&\text{if }\rho\neq 0,\\
\int_{\mathbb{P}}\log_YAd\nu(A),&\text{if }\rho=0.
\end{cases}
\end{equation}
Notice that for $\rho\neq 0$ we have
\begin{equation}\label{eq2:T:ResolventConv}
\begin{split}
F(\rho,Y)&=\frac{1}{\rho}\left\{\left(\cdots(Y\#_{\frac{\rho}{\rho+n}}A_n)\#_{\frac{\rho}{\rho+n}}\cdots)\#_{\frac{\rho}{\rho+n}}A_1\right)\#_{\frac{\rho}{\rho+\lambda}}X-Y\right\}\\
&=\frac{1}{\rho}[Y_n\#_{\frac{\rho}{\rho+\lambda}}X-Y_n]+\sum_{i=0}^{n-1}\frac{1}{\rho}[Y_i\#_{\frac{\rho}{\rho+n}}A_{n-i}-Y_i]
\end{split}
\end{equation}
where $Y_0:=Y$ and $Y_{i+1}:=Y_i\#_{\frac{\rho}{\rho+n}}A_{n-i}$ for $0\leq i\leq n-1$. Thus as $\rho\to 0$ we have $Y_i\to Y$ and by similar calculations on each summand in the above as in the proof of Lemma~\ref{L:PowerCont}, we get that
\begin{equation*}
\lim_{\rho\to 0}F(\rho,Y)=\frac{1}{\lambda}\log_YX+\sum_{i=1}^{n}\frac{1}{n}\log_YA_i
\end{equation*}
in the norm topology. The convergence of the Fr\'echet derivative with respect to $Y$ is a bit more delicate calculation starting from \eqref{eq2:T:ResolventConv} using induction to calculate the derivatives of the $Y_i$ defined recursively by composition of geometric means $\#_{\frac{\rho}{\rho+n}}$, but the principles are the same as in the proof of Lemma~\ref{L:PowerCont} and the calculation is left to the reader.

Now the remaining part of the proof follows the lines of the corresponding part of the proof of Theorem~\ref{T:PowerNormCont}.
\end{proof}

\section{A Continuous-time law of large numbers for $\Lambda$}
Here we combine the results of the previous sections to obtain convergence theorems valid for the nonlinear semigroups solving the Cauchy problem in Theorem~\ref{T:StrongSolution}.

\begin{theorem}\label{T:contlln}
Let $\mu\in\mathcal{P}^1(\mathbb{P})$ and let $\{Y_i\}_{i\in\mathbb{N}}$ be a sequence of independent, identically distributed $\mathbb{P}$-valued random variables with law $\mu$. Let $\mu_n:=\sum_{i=1}^n\frac{1}{n}\delta_{Y_i}\in\mathcal{P}^1(\mathbb{P})$ denote the empirical measures. Let $S^{\mu}(t)$ and $S^{\mu_n}(t)$ denote the semigroups corresponding to the resolvents $J^{\mu}_\lambda$ and $J^{\mu_n}_\lambda$ according to Theorem~\ref{T:ExponentialFormula} for $t>0$. Then almost surely
\begin{equation}\label{eq:T:contlln1}
\lim_{n\to\infty}S^{\mu_n}(t)=S^{\mu}(t)
\end{equation}
uniformly in $d_\infty$ on compact time intervals.

Moreover let $F^{\mu_n}_{\rho}:=J_{\rho/n}^{\delta_{Y_n}}\circ\cdots\circ J_{\rho/n}^{\delta_{Y_1}}$ where $J_{\rho}^{\delta_{A}}(X):=X\#_{\frac{\rho}{\rho+1}}A$ in the spirit of \eqref{eq:D:resolvent}. Then almost surely
\begin{equation}\label{eq:T:contlln2}
\lim_{n\to\infty}\lim_{m\to\infty}(F^{\mu_n}_{t/m})^m=S^{\mu}(t)
\end{equation}
uniformly in $d_\infty$ on compact time intervals.
\end{theorem}
\begin{proof}
Let $X\in\mathbb{P}$. By Proposition~\ref{P:separableSupp} the $\supp(\mu)$ is separable and $\supp(\mu_n)\subseteq \supp(\mu)$. Thus by Varadarajan's theorem \cite{villani} for empirical barycenters on the complete Polish metric space $(\supp(\mu),d_\infty)$, the sequence $\mu_n$ converges weakly to $\mu$ almost surely and then by Proposition~\ref{P:weakW1agree} we get $W_1(\mu_n,\mu)\to 0$ almost surely. Then by Theorem~\ref{T:LambdaExists} we have that $J^{\mu_n}_\lambda(Z)\to J^{\mu}_\lambda(Z)$ almost surely in $d_\infty$ for any $Z\in\mathbb{P}$. By Theorem~\ref{T:ExponentialFormula} we have that $S^{\mu}(t)X:=\lim_{m\to\infty}\left(J_{t/m}^{\mu}\right)^{m}(X)$ and $S^{\mu_n}(t)X:=\lim_{m\to\infty}\left(J_{t/m}^{\mu_n}\right)^{m}(X)$ uniformly on compact time intervals. The estimate
\begin{equation*}
\begin{split}
d_\infty(\left(J^{\mu_n}_{t/m}\right)^{m}(X),\left(J^{\mu}_{t/m}\right)&^{m}(X))\leq d_\infty(\left(J^{\mu_n}_{t/m}\right)^{m}(X),\left(J^{\mu_n}_{t/m}\right)^{m-1}J^{\mu}_{t/m}(X))\\
& +d_\infty(\left(J^{\mu_n}_{t/m}\right)^{m-1}J_{t/m}^\mu(X),\left(J_{t/m}^\mu\right)^{m-1}J_{t/m}^\mu(X))
\end{split}
\end{equation*}
by induction yields the existence of $n_0\in\mathbb{N}$, such that for $n>n_0$ we get that
\begin{equation*}
d_\infty(\left(J^{\mu_n}_{t/m}\right)^{m}(X),\left(J_{t/m}^\mu\right)^{m}(X))<\epsilon
\end{equation*}
for a fixed $m$, thus by the triangle inequality we get \eqref{eq:T:contlln1} almost surely. Uniform convergence can be showed similarly along the lines of the proof of Proposition~\ref{P:TrotterApprox}.

Now \eqref{eq:T:contlln2} follows from \eqref{eq:T:contlln1} combined with Theorem~\ref{T:ResolventConv} and Theorem~\ref{T:TrotterFormula}.
\end{proof}

\subsection*{Acknowledgments}
 The work of Y.~Lim was supported by the National
Research Foundation of Korea (NRF) grant funded by the Korea
government(MEST) No.2015R1A3A2031159. The work of M.~P\'alfia was
supported in part by the "Lend\"ulet" Program (LP2012-46/2012) of the Hungarian Academy of Sciences and the National Research Foundation of Korea (NRF) grant funded by the Korea government(MEST) No.2016R1C1B1011972.

\end{document}